\documentclass[12pt,a4paper]{article}
\usepackage{fancyhdr}
\usepackage{amsthm}
\usepackage{amsmath}
\usepackage{amsmath, amssymb, latexsym}
\usepackage{epsfig}

\usepackage[top=2.5cm,bottom=2.5cm,left=2.5cm,right=2.5cm]{geometry}
\usepackage[colorlinks=true,urlcolor=blue,pdfstartview=FitV, linkcolor=blue,citecolor=blue,bookmarksopen, bookmarksnumbered={true}]{hyperref}

\newtheorem{theorem}{Theorem}[section]
\newtheorem{definition}{Definition}[section]

\newtheorem{example}{Example}[section]
\newtheorem{remark}{Remark}[section]
\newtheorem{proposition}{Proposition}[section]
\newtheorem{lemma}{Lemma}[section]
\title{Darboux problem for Caputo-Katugampola fuzzy fractional
	order differential equations}
\author{}
\date{}
\begin{document}
	\maketitle
	\hrule
	\begin{center}
		\textbf{Nagwa A. Saeed$^ {1,2*} $ and Deepak B. Pachpatte$^ {1} $}\\
		\begin{flushleft}
			1 Department of Mathematics, Dr. Babasaheb Ambedkar Marathwada University, Aurangabad, India.\\
			2 Department of Mathematics, Taiz University, Taiz, Yemen.
		\end{flushleft}
	\end{center}
	\hrule
	
	\section*{Abstract}
	\ \ \ In this paper, we investigate existence and uniqueness of solutions for Darboux type problem for fuzzy fractional order  differential equation. We used Caputo-Katogampola fuzzy fractional derivative for proving our results. Schauder's fixed point theorem is used in proving our results. some applications are also provided to give the usefulness of our results.\\
	\textbf{Keywords:}
	Caputo-Katugampola fractional derivative, Fuzzy fractional partial derivative, Fuzzy valued function, Fixed point theorem, partial differential equation.\\
	\textbf{Mathematics Subject Classification (2020)}
	 26A33, 35R11, 35R13. 
	\section{Introduction}\label{S:DS}
	\ \ \ The significance of fractional calculus and its potential applications has increased significantly due to its effectiveness in accurately modeling complex phenomena across a wide range of scientific and engineering fields. These fields include aerodynamics, control systems, signal processing, bioengineering and biomedical sciences. Fractional differential equations and dynamical systems have proven to be valuable tools for modeling various phenomena in engineering, physics and economics. Moreover, fractional calculus finds extensive applications in areas such as viscoelasticity, heat conduction in materials with memory and fluid dynamic traffic modeling (see \cite{r1,r2,r3,r4}).
	
	 In \cite{r5,r6} a novel concept of fractional calculus called Caputo-Katugampola ($\mathcal{CK}$) fractional integral and derivative was introduced. Building upon this new fractional derivative, the existence and uniqueness of solutions for fractional differential equations with $\mathcal{CK}$ derivative were investigated in \cite{r7} using Schauder's second fixed point theorem. In \cite{r8} a discrete version of the $\mathcal{CK}$ derivative was proposed and a numerical formula was derived to solve a linear fractional differential equation with $\mathcal{CK}$ fractional derivative. The investigation of chaotic behavior and stability results of fractional differential equations with $\mathcal{CK}$ derivative was carried out in \cite{r9}. A Darboux problem related to a fractional hyperbolic integro-differential inclusion, characterized by the $\mathcal{CK}$ fractional derivative, is examined by A. Cernea in \cite{r10}. In the year 2020, researchers  conducted a study focusing on the existence and uniqueness of solutions for the Darboux problem of partial differential equations involving the $\mathcal{CK}$ fractional derivative (see \cite{r11}).
	
	 Recently, there has been significant development in fuzzy  calculus across various domains. Bede et al. \cite{r12,r13,r14} have presented numerous applications of this calculus. The theory of fuzzy calculus, along with its advantages, has attracted the attention of many mathematicians, leading to the development of research studies on existence, uniqueness and stability for fractional differential equations see \cite{r15}-\cite{r25}. Recently, Fan Zhang et al. \cite{r26} proved the existence of a weak solutions for a class of new coupled systems of fuzzy fractional partial differential equations in the sense of Caputo $g\mathcal{H}$-derivative by using Schauder fixed point
	 theorem. 
	 
	 In \cite{r27}, the authors put forward the concept of the $\mathcal{CK}$ fuzzy fractional derivative. They proceeded to investigate the existence and uniqueness of solutions for the initial value problem of fractional fuzzy differential equations utilizing this innovative derivative. In reference \cite{r28}, the authors conducted an investigation into the existence, uniqueness and various forms of Ulam-Hyers stability for solutions of an impulsive coupled system of fractional differential equations. They achieved this by employing the $\mathcal{CK}$ fuzzy fractional derivative. Furthermore, the Perov-type fixed point theorem was applied to establish the existence and uniqueness of the proposed system. Motivated by the previously mentioned works, the primary objective of our paper is two-fold:
	 \begin{enumerate}
	 	\item Extend the concept of the fuzzy $\mathcal{CK}$ derivatives for one-variable functions, as presented in reference \cite{r27}, to be applicable for fuzzy-valued multivariable functions and we examine the existence and uniqueness of solution for a single initial value problem concerning $\mathcal{CK}$ fuzzy fractional differential equation of the following form
	\begin{equation}
		\label{i1}
		\left\{
		\begin{array}{cc}
			^{C}\mathfrak{D}_{0^{+}}^{\varphi,\rho}\upsilon(\varkappa,y) = \mathcal{F}(\varkappa,y,\upsilon(\varkappa,y)),  \\\\
			\upsilon(\varkappa,0)=\xi_{1}(\varkappa),~~
			\upsilon(0,y)=\xi_{2}(y),
		\end{array}\right .
	\end{equation}
	where $\varphi=(\varphi_{1},\varphi_{2}) \in(0,1]$,  	
	is the fractional order of $\mathcal{CK}$ $g\mathcal{H}$-fractional derivative operator $^{C}\mathfrak{D}_{0^{+}}^{\varphi,\rho}$ defined later in Definition (\ref{i14}),
	$(\varkappa,y)\in \boldsymbol{\varOmega},$ $\mathcal{F}: \boldsymbol{\varOmega} \times \hat{E_{f}} \rightarrow \hat{E_{f}},$ $\xi_{1}:(0,a]\rightarrow \hat{E_{f}}$, $\xi_{2}:(0,b]\rightarrow \hat{E_{f}}$ and $\boldsymbol{\varOmega} = (0,a] \times (0,b].$
	\item By using some standard fixed point theorem techniques we establish existence and uniqueness results of solutions for the following coupled system of fuzzy fractional differential equations involving the $\mathcal{CK}$ $g\mathcal{H}$-derivative:
	\end{enumerate}
		\begin{equation}
		\label{i2}
		\left\{
		\begin{array}{cc}
		^{C}\mathfrak{D}_{0^{+}}^{\varphi,\rho}\upsilon(\varkappa,y) = \mathcal{F}(\varkappa,y,\upsilon(\varkappa,y),\omega(\varkappa,y)),\\\\
		^{C}\mathfrak{D}_{0^{+}}^{\psi,\rho}\omega(\varkappa,y,) = \mathcal{G}(\varkappa,y,\upsilon(\varkappa,y),\omega(\varkappa,y)),
		\end{array}\right .
	\end{equation}
with the following initial conditions:
\begin{equation}
	\label{i3}
	\left\{
	\begin{array}{cc}
		\upsilon(\varkappa,0)=\xi_{1}(\varkappa),~~ 
		\upsilon(0,y)=\xi_{2}(y),\\
		\omega(\varkappa,0)=\eta_{1}(\varkappa),~~ 
		\omega(0,y)=\eta_{2}(y),
	\end{array}\right .
\end{equation}
	where $\varphi=(\varphi_{1},\varphi_{2}),\psi=(\psi_{1},\psi_{2})\in(0,1]\times(0,1]$, and $\mathcal{CK}$ $g\mathcal{H}$-fractional derivative operators $^{C}\mathfrak{D}_{0^{+}}^{\varphi,\rho} $ and $^{C}\mathfrak{D}_{0^{+}}^{\psi,\rho}$ are the same as in (\ref{i1}).  
	\section{Preliminaries}
In this section, we introduce fundamental definitions and concepts within the realm of fuzzy set theory and fuzzy fractional calculus. These notions will serve as a foundation for our subsequent discussion.\\
	
Let $E_{f}$ be the space of all fuzzy numbers on $\mathbb{R}$, $\hat{E_{f}}$ is a space of fuzzy number $\mathfrak{z}\in E_{f}$, which has the property that the function $r \mapsto [\mathfrak{z}]^{r}$ is continuous with respect to Hausdorff metric on [0, 1],
	$C(\boldsymbol{\varOmega}, {\hat{E_{f}}})$ denotes a space of all continuous fuzzy-valued functions which are  on $\boldsymbol{\varOmega}$, $AC(\boldsymbol{\varOmega}, {\hat{E_{f}}})$ denotes the set of all  absolutely continuous fuzzy-valued functions and $L[\boldsymbol{\varOmega}, {\hat{E_{f}}}]$ by the set of Lebesque integrable for fuzzy-valued functions on $\boldsymbol{\varOmega}$.
	
	The definition of a fuzzy number and its r-level form can be described as follows:
	\begin{definition} (\cite{r29})\label{2d1}
	A fuzzy number can be defined as a mapping $\mathfrak{z}:\mathbb{R}\rightarrow [0,1]$ with the following characteristics:
		\begin{description}
			\item[(1)] The function $\mathfrak{z}$ is normal, where $\mathfrak{z}(\varkappa_{0})=1$ for a given $\varkappa_{0} \in \mathbb{R}$.
			\item[(2)] The function $\mathfrak{z}$ is convex for $\varkappa_{1},\varkappa_{2},\mathfrak{y} \in \mathbb{R}$ and $t \in [0, 1]$, satisfying the inequality\\
			$$\mathfrak{z}(t\varkappa_{1}+(1-t)\varkappa_{2})\geq \min\{\mathfrak{z}(\varkappa_{1}),\mathfrak{z}(\varkappa_{2})\}.$$ 
			\item[(3)] $\mathfrak{z}$ is semicontinuous. 
		    \item[(4)] $\mathfrak{z}$ is a compactly supported on $\mathbb{R}$.
		\end{description}
		The set of a fuzzy number $\mathfrak{z}(\varkappa)\in E_{f}$ in the r-level form is
		denoted by $[\mathfrak{z}]^{r}$ and defined as:
		\[[\mathfrak{z}]^{r}=\{\varkappa \in \mathbb{R}|\mathfrak{z}(\varkappa) \geq r\}=[\underline{\mathfrak{z}_{r}},\overline{\mathfrak{\mathfrak{z}}_{r}}], 0 \leq r \leq 1).\]
		It is clear that the set of a fuzzy number $\mathfrak{z}$ in r-level form is  a closed and bounded interval $[\underline{\mathfrak{z}_{r}},\overline{\mathfrak{\mathfrak{z}}_{r}}]$, where $\underline{\mathfrak{z}_{r}}$ is the left end point and $\overline{\mathfrak{\mathfrak{z}}_{r}}$ is the right end point. The diameter of the r-level set of $\mathfrak{z}$ is denoted by $diam[\mathfrak{z}]^{r}$ and defined as:\\ \[diam[\mathfrak{z}]^{r}=\overline{\mathfrak{\mathfrak{z}}_{r}}-\underline{\mathfrak{z}_{r}}.\]
	\end{definition}
	\begin{definition}(\cite{r22,r32})\label{2d3}
		Let ${\mathbb{D}}: E_{f}\times E_{f} \longrightarrow \mathbb{R}$
		be the Hausdorff distance between two fuzzy numbers p,q and defined as
		\begin{equation}
			\begin{split}
					{\mathbb{D}}(\mathfrak{p},\mathfrak{q})&=\sup_{0 \leq r \leq 1} d_H \{[p]^{r},[q]^{r}\}\\&=\sup_{0 \leq r \leq 1} max\{| \underline{\mathfrak{p}_{r}}-\underline{\mathfrak{q}_{r}}|,|\ \overline{\mathfrak{p}_{r}}-\overline{\mathfrak{q}_{r}}|\}.
			\end{split}	
		\end{equation}
	The supremum metric $H^{*}$ on $C(\boldsymbol{\varOmega},E_{f})$ is considered by
	\begin{equation}	\label{2d5}
		{H^{*}}(p,q) = \sup_{(\varkappa,y) \in I}\{\
		{\mathbb{D}}(p(\varkappa,y),q(\varkappa,y))\}.
	\end{equation}	
			\end{definition}
		According to [29,31,33], fuzzy number spaces $E_{f}$ and  $\hat{E_{f}}$ are semilinear spaces having the cancellation property. Furthermore, it is proved that fuzzy numbers spaces $E_{f}$ and $\hat{E_{f}}$ with metric $\mathbb{D}$ are complete metric semilinear spaces. Hence, the fuzzy-valued continuous functions space $C(\boldsymbol{\varOmega},\hat{E_{f}})$ is complete metric semilinear space, which is a Banach semilinear space with the cancellation property.
		\begin{definition}\cite{r30}
			Let $\mathfrak{p},\mathfrak{q} \in \hat{E_{f}}$. If there exists $\boldsymbol{w} \in \hat{E_{f}}$ such that $\mathfrak{p} =\mathfrak{q} + \boldsymbol{w}$, then $\boldsymbol{w}$
			is called the Hukuhara difference of $\mathfrak{p}$ and $\mathfrak{q}$ and it is denoted by $\mathfrak{p}\ominus\mathfrak{q}$.
		\end{definition}
		In \cite{r22} authors have given some  properties of the metric ${\mathbb{D}}$ in $\hat{E_{f}}$ and Hukuhara difference as:
		\begin{lemma}\cite{r22}
		For all $\mathfrak{p}, \mathfrak{q}, \boldsymbol{l}, \boldsymbol{k}, \boldsymbol{g} \in \hat{E_{f}}$	we have
		\begin{itemize}
			\item[$(1)$] ${\mathbb{D}}(\mathfrak{p} + \boldsymbol{l},\mathfrak{q}+ \boldsymbol{l})={\mathbb{D}}(\mathfrak{p},\mathfrak{q}),$
			\item [$(2)$] ${\mathbb{D}}(\mathfrak{p} + \mathfrak{q},\boldsymbol{k} + \boldsymbol{g})\leq {\mathbb{D}}(\mathfrak{p},\boldsymbol{k})+{\mathbb{D}}(q,\boldsymbol{g}),$
			\item [$(3)$] ${\mathbb{D}}(\mathfrak{p} + \mathfrak{q},0)= {\mathbb{D}}(\mathfrak{u},0)+{\mathbb{D}}(\mathfrak{q},0),$
			\item [$(4)$] If $\mathfrak{p} \ominus \mathfrak{q}$ exists then $(-1)\mathfrak{p}\ominus(-1)\mathfrak{q}$ exists and $(-1)(\mathfrak{p}\ominus\mathfrak{q}) = (-1)\mathfrak{p} \ominus	(-1)\mathfrak{q},$
			\item [$(5)$] If $\mathfrak{p}\ominus \mathfrak{q}$ and $\boldsymbol{k} \ominus \boldsymbol{g}$ exist then ${\mathbb{D}}(\mathfrak{p} \ominus \mathfrak{q},\boldsymbol{k}\ominus \boldsymbol{g}) \leq {\mathbb{D}}(\mathfrak{p},\boldsymbol{k}) +{\mathbb{D}}(\mathfrak{q},\boldsymbol{g})$.
		\end{itemize}
	\end{lemma}
\begin{definition}(\cite{r13})\label{2d4}
	The generalized Hukuhara difference of two fuzzy numbers $\mathfrak{p},\mathfrak{q} \in \hat{E_{f}}$
	($g\mathcal{H}$-difference for short) is defined as follows:
	given by
	\begin{eqnarray*}
		\mathfrak{p} \ominus_{g\mathcal{H}}\mathfrak{q}=\boldsymbol{w} \Leftrightarrow
		(i)~ \mathfrak{p}=\mathfrak{q}+ \boldsymbol{w}~~
		or~~ (ii)~ \mathfrak{q}=\mathfrak{p} +(-1)\boldsymbol{w}.
	\end{eqnarray*}
\end{definition}
A function $\upsilon:\boldsymbol{\varOmega} \longrightarrow \hat{E_{f}}$ is called d-increasing (d-decreasing) on $\boldsymbol{\varOmega}$ if for every $r \in [0, 1]$
the function $(\varkappa,y) \mapsto diam[\upsilon(\varkappa,y,r)]$ is nondecreasing (nonincreasing) on $\boldsymbol{\varOmega}$. If $\upsilon$ is d-increasing or
d-decreasing on $\boldsymbol{\varOmega}$, then we say that $\upsilon$ is d-monotone on $\boldsymbol{\varOmega}$.

The definitions of integration and partial derivative incorporating fuzzy concepts were introduced by authors in \cite{r14,r22, r31} as follows:
	\begin{definition}(\cite{r14})\label{2d6}
		Let $(\varkappa_{0}, \mathfrak{y}_{0})\in J\subset \mathbb{R}^{2}$, then the g$\mathcal{H}$-partial derivative in first order of a fuzzy mapping
		$\mathcal{U}:J \longrightarrow \hat{E_{f}}$ at $(\varkappa_{0}, \mathfrak{y}_{0})$ with respect to $\varkappa$ is the function $\frac{\partial \mathcal{U}(\varkappa_{0},\mathfrak{y}_{0})}{\partial \varkappa}$ 
		given by
		\[\frac{\partial \mathcal{U}(\varkappa_{0},\mathfrak{y}_{0})}{\partial \varkappa}=\lim_{h\to 0} \frac{\mathcal{U}(\varkappa_{0}+h,\mathfrak{y}_{0})\ominus_{g\mathcal{H}}\mathcal{U}(\varkappa_{0},\mathfrak{y}_{0})}{h},\]
		and  with respect to $\mathfrak{y}$ is the function $\frac{\partial \mathcal{U}(\varkappa_{0},\mathfrak{y}_{0})}{\partial \mathfrak{y}}$  given by
		\[\frac{\partial \mathcal{U}(\varkappa_{0},\mathfrak{y}_{0})}{\partial \mathfrak{y}}=\lim_{k\to 0} \frac{\mathcal{U}(\varkappa_{0},\mathfrak{y}_{0}+k)\ominus_{g\mathcal{H}}\mathcal{U}(\varkappa_{0},\mathfrak{y}_{0})}{k},\]
		provided that the g$\mathcal{H}$-differences $f(\varkappa_{0}+h,\mathfrak{y}_{0})\ominus_{g\mathcal{H}}\mathcal{U}(\varkappa_{0},\mathfrak{y}_{0}), \mathcal{U}(\varkappa_{0},\mathfrak{y}_{0}+k)\ominus_{g\mathcal{H}}\mathcal{U}(\varkappa_{0},\mathfrak{y}_{0})$ exist and $\frac{\partial \mathcal{U}(\varkappa_{0},\mathfrak{y}_{0})}{\partial \varkappa}$, $\frac{\partial \mathcal{U}(\varkappa_{0},\mathfrak{y}_{0})}{\partial \mathfrak{y}}$ in $\hat{E_{f}}$.\\
		The g$\mathcal{H}$-partial derivatives of higher order of  $\mathcal{U}$ at the point $(\varkappa_{0}, \mathfrak{y}_{0}) \in J$ are defined similarly.
	\end{definition}
	\begin{definition}({\cite{r14,r31}})\label{2d7}
		Let $\mathcal{U}:J \longrightarrow \hat{E_{f}}$ be partial g$\mathcal{H}$-differentiable with respect to $\varkappa$ at $(\varkappa_{0}, \mathfrak{y}_{0}) \in J$.
		Then
	\begin{enumerate}
		\item $\mathcal{U}$ is $(i)$-g$\mathcal{H}$ differentiable with respect to $\varkappa$ at $(\varkappa_{0}, \mathfrak{y}_{0}) \in J$. If 
		\[ \Bigl[\frac{\partial \mathcal{U}(\varkappa_{0},\mathfrak{y}_{0})}{\partial \varkappa} \Bigl]^{r} = \Bigl[\frac{\partial \underline{\mathcal{U}}(\varkappa_{0},\mathfrak{y}_{0})}{\partial \varkappa},\frac{\partial \overline{\mathcal{U}}(\varkappa_{0},\mathfrak{y}_{0})}{\partial \varkappa} \Bigl],~~~~~\forall r \in [0, 1], \]
		\item $\mathcal{U}$ is $(ii)$-g$\mathcal{H}$ differentiable with respect to $\varkappa$ at $(\varkappa_{0}, \mathfrak{y}_{0}) \in J$. If
		\[  \Bigl[\frac{\partial \mathcal{U}(\varkappa_{0},\mathfrak{y}_{0})}{\partial \varkappa} \Bigl]^{r} = \Bigl[\frac{\partial \overline{\mathcal{U}}(\varkappa_{0},\mathfrak{y}_{0})}{\partial \varkappa},\frac{\partial \underline{\mathcal{U}}(\varkappa_{0},\mathfrak{y}_{0})}{\partial \varkappa} \Bigl],~~~~~\forall
		r \in [0, 1]. \]
	\end{enumerate}
		Similarly, we can defined the types of fuzzy g$\mathcal{H}$-derivatives of $\mathcal{U}$ with respect to $\mathfrak{y}$ and higher order of $\mathcal{U}$ at the point $(\varkappa_{0},\mathfrak{y}_{0}) \in J$.
	\end{definition}
	\begin{definition}(\cite{r22,r31})\label{2d8}
	For any fixed point $\varkappa_{0}$ and $(\varkappa_{0}, \mathfrak{y}) \in J$, if in any neighborhood V of $(\varkappa_{0}, \mathfrak{y}) \in J$, there exist points ${A_{1}}(\varkappa_{1}, \mathfrak{y}),{A_{2}}(\varkappa_{2}, \mathfrak{y})$ whereas $\varkappa_{1} < \varkappa_{0} < \varkappa_{2}$  such that 
	\begin{enumerate}
	\item  $\mathcal{F}$ is $(i)$-g$\mathcal{H}$ differentiable at ${A_{1}}$ whereas $\mathcal{F}$  is $(ii)$-g$\mathcal{H}$ differentiable at ${A_{2}}$ for all $\mathfrak{y}$, or\\
	\item  $\mathcal{F}$ is $(i)$-g$\mathcal{H}$ differentiable at ${A_{2}}$ while $\mathcal{F}$  is $(ii)$-g$\mathcal{H}$ differentiable at ${A_{1}}$ for all $\mathfrak{y}$.
	Then,$(\varkappa_{0}, \mathfrak{y}) \in J$ is  known as a switching point for the differentiability of $\mathcal{F}$ with
	respect to $\varkappa$.
\end{enumerate}
\end{definition}
\begin{definition}\cite{r22}
	Let $\mathcal{U}(.,.) \in C(J,\hat{E_{f}})$ such that $\frac{\partial \mathcal{U}(x,y) }{\partial x} \in C(J,\hat{E_{f}})$ and they don’t have any switching point on $J$. Denote $\frac{\partial^{2}\mathcal{U}(x_{0}, y_{0})}{\partial x \partial y}$ by the mixed second order gH-partial derivative of f with respect to x and y at $(x_{0}, y_{0})$ $\in I$. We say that
	\begin{itemize}
	  \item[$(1)$]  $\frac{\partial^{2}\mathcal{U}(.,.)}{\partial x \partial y}$ is a $(\hat{i})$-g$\mathcal{H}$ differentiable at $(x_{0}, y_{0}
	)$, if the type of g$\mathcal{H}$-differentiability at $(x_{0}, y_{0})$of both $\mathcal{U}(.,.)$ and $\frac{\partial \mathcal{U}(x,y) }{\partial x}$ are the same.
   	\item[$(2)$]  $\frac{\partial^{2}\mathcal{U}(.,.)}{\partial x \partial y}$ is a $(\widehat{ii})$-g$\mathcal{H}$ differentiable at $(x_{0}, y_{0}) \in I$, if the type of g$\mathcal{H}$-differentiability at $(x_{0}, y_{0})$ of both $\mathcal{U}(.,.)$ and $\frac{\partial \mathcal{U}(x,y) }{\partial x}$ are different.
\end{itemize}
\end{definition} 
	\begin{lemma} \cite{r22}\label{2l1}
	Let $\mathcal{F}\in C(\mathbb{R}^{2},\hat{E_{f}}).$
	\begin{itemize}
		\item[$(i)$]
	 If $\mathcal{F}$ is $(i)$-g$\mathcal{H}$ differentiable with respect to y, with no switching point on $\mathbb{R} \times [b, \mathfrak{y}]$ then $\frac{\partial \mathcal{F}(\varkappa,\mathfrak{y})}{\partial \mathfrak{y}}$
	is integrable
	on $[b, \mathfrak{y}]$ and
	\[\int_{b}^{\mathfrak{y}} \ \frac{\partial \mathcal{F}(\varkappa,s)}{\partial s} \, ds = \mathcal{F}(\varkappa,\mathfrak{y}) \ominus \mathcal{F}(\varkappa,b).\]
	\item[$(ii)$]
	 If $\mathcal{F}$ is $(ii)$-g$\mathcal{H}$ differentiable with respect to y, with no switching point on $\mathbb{R} \times [b, \mathfrak{y}]$ then  $\frac{\partial \mathcal{F}(\varkappa,\mathfrak{y})}{\partial \mathfrak{y}}$	is integrable
	on $[b, \mathfrak{y}]$ and
	\[\int_{b}^{\mathfrak{y}} \ \frac{\partial \mathcal{F}(\varkappa,s)}{\partial s} \, ds = (-1)\mathcal{F}(\varkappa,b) \ominus  (-1)\mathcal{F}(\varkappa,\mathfrak{y}).\]
		\end{itemize}
\end{lemma}
\begin{remark}
 If the diameter function of $\upsilon(\varkappa,y,r)$ $\big(diam[\upsilon(\varkappa,y)]^{r}\big)$ is d-increasing (d-decreasing) for all $r$, then we
have $(i)$-g$\mathcal{H}$ differentiable ($(ii)$-g$\mathcal{H}$ differentiable).
\end{remark}
  The authors in \cite{r11} introduced the concept of mixed Katugampola fractional integral of order $\varphi=(\varphi_{1},\varphi_{2})\in(0,1)^{2}$ for real-valued functions $\mu \in L(\boldsymbol{\varOmega} , \mathbb{R})$ as follows.
  \begin{equation}\label{i4}
  \mathfrak{I}_{0^{+}}^{\varphi,\rho}\mu (\varkappa,y)=\frac{\rho_{1}^{1-\varphi_{1}}\rho_{2}^{1-\varphi
  		_2}}{\Gamma(\varphi_{1})\Gamma(\varphi_{2})} \int_{0}^{\varkappa}\int_{0}^{y}\frac{s^{\rho_{1}-1}t^{\rho_{2}-1}}{(\varkappa^{\rho_{1}}-s^{\rho_{1}})^{1-\varphi_{1}}(y^{\rho_{2}}-t^{\rho_{2}})^{1-\varphi_{2}}} \mu(s,t)dtds,
  	\end{equation}
  for $(\varkappa, y) \in \boldsymbol{\varOmega}$, $\rho=(\rho_{1},\rho_{2})>0.$ Also they defined the mixed Katugampola fractional derivative of order $\varphi=(\varphi_{1},\varphi_{2})\in(0,1)^{2}$ for a real function $\mu \in AC(\boldsymbol{\varOmega}, \mathbb{R})$ by
  \[\mathfrak{D}_{0^+}^{\varphi,\rho}\mu(\varkappa,y) =\varkappa^{1-\rho_{1}}y^{1-\rho_{2}}\mathfrak{D}^{2}_{\varkappa,y}( \mathfrak{I}_{0^{+}}^{1-\varphi,\rho}\mu (\varkappa,y)),\]
  	where $\mathfrak{D}^{2}_{\varkappa,y}= \frac{\partial
  		^{2}}{\partial\varkappa \partial y}$ and
   	\[\mathfrak{I}_{0^{+}}^{1-\varphi,\rho}\mu (\varkappa,y) = \frac{\rho_{1}^{\varphi_{1}}\rho_{2}^{\varphi
   			_2}}{\Gamma(1-\varphi_{1})\Gamma(1-\varphi_{2})} \int_{0}^{y}\frac{s^{\rho_{1}-1}t^{\rho_{2}-1}}{(\varkappa^{\rho_{1}}-s^{\rho_{1}})^{\varphi_{1}}(y^{\rho_{2}}-t^{\rho_{2}})^{\varphi_{2}}} \mu(s,t)dtds.\]
   	If $\mu \in L(\boldsymbol{\varOmega} , \mathbb{R})$ is a real function such that
   $\mathfrak{D}_{0^+}^{\varphi,\rho}\mu(\varkappa,y)$ exists on $\boldsymbol{\varOmega}$, then the $\mathcal{CK}$ fractional
   	derivative
   	$^{C}\mathfrak{D}_{0^+}^{\varphi,\rho}\mu(\varkappa,y)$
   	of order $\varphi=(\varphi_{1},\varphi_{2})\in(0,1)^{2}$ is defined by
   	\[^{C}\mathfrak{D}_{0^+}^{\varphi,\rho}\mu(\varkappa,y)=\mathfrak{D}_{0^+}^{\varphi,\rho}(\mu(\varkappa,y)-\mu(\varkappa,0)-\mu(0,y)+\mu(0,0)).\]
   	Therefore, we have
   	\begin{equation}\label{i5}
   	^{C}\mathfrak{D}_{0^+}^{\varphi,\rho}\mu(\varkappa,y)=\mathfrak{D}_{0^+}^{\varphi,\rho}\upsilon(\varkappa,y)-\mathfrak{D}_{0^+}^{\varphi,\rho}\mu(\varkappa,0)-\mathfrak{D}_{0^+}^{\varphi,\rho}\mu(0,y)+\frac{\upsilon(0,0)\rho_{1}^{\varphi_{1}} \rho_{2}^{\varphi_{2}}}{\Gamma(1-\varphi_{1})\Gamma(1-\varphi_{2})}\varkappa^{-\varphi_{1}}y^{-\varphi_{2}}.
   \end{equation}
   	If $\mu \in AC[a, b],$ then we have (see lemma (2.1) in \cite{r11})
   	\[^{C}\mathfrak{D}_{0^+}^{\varphi,\rho}\mu(\varkappa,y)= \frac{\rho_{1}^{\varphi_{1}}\rho_{2}^{\varphi
   	_2}}{\Gamma(1-\varphi_{1})\Gamma(1-\varphi_{2})} \int_{0}^{\varkappa}\int_{0}^{y}\frac{\mathfrak{D}^{2}_{s,t}\mu(s,t)}{(\varkappa^{\rho_{1}}-s^{\rho_{1}})^{\varphi_{1}}(y^{\rho_{2}}-t^{\rho_{2}})^{\varphi_{2}}} dtds.\]
   Now we proceed to establish the definition of the Katugampola fractional integral and derivative for the fuzzy-valued multivariable function $\upsilon:\boldsymbol{\varOmega} \rightarrow E_{f}$ as folllows.
	\begin{definition}
		Let  $\varphi=(\varphi_{1},\varphi_{2})\in(0,1)^{2}$ and $\upsilon \in L(\boldsymbol{\varOmega},\hat{E_{f}}).$ The mixed Katugampola fuzzy fractional integral is defined by
			\[\mathfrak{I}_{0^{+}}^{\varphi,\rho}\upsilon (\varkappa,y)=\frac{\rho_{1}^{1-\varphi_{1}}\rho_{2}^{1-\varphi
				_2}}{\Gamma(\varphi_{1})\Gamma(\varphi_{2})} \int_{0}^{\varkappa}\int_{0}^{y}\frac{s^{\rho_{1}-1}t^{\rho_{2}-1}}{(\varkappa^{\rho_{1}}-s^{\rho_{1}})^{1-\varphi_{1}}(y^{\rho_{2}}-t^{\rho_{2}})^{1-\varphi_{2}}}  \upsilon(s,t)dtds,\]
		for $(\varkappa, y) \in \boldsymbol{\varOmega}$, $\rho=(\rho_{1},\rho_{2})>0.$ Since $[\upsilon(\varkappa,y)]^{r} = [\underline{\upsilon}(\varkappa,y,r), \overline{\upsilon}(\varkappa,y,r)]$ and $r\in [0,1],$ by r-level form 
		\[[\mathfrak{I}_{0^{+}}^{\varphi,\rho}\upsilon(\varkappa,y)]^{r} = [\mathfrak{I}_{0^{+}}^{\varphi,\rho}\underline{\upsilon}(\varkappa,y,r),\mathfrak{I}_{0^{+}}^{\varphi,\rho}\overline{\upsilon}(\varkappa,y,r)],\]
		where $\mathfrak{I}_{0^{+}}^{\varphi,\rho}\underline{\upsilon}(\varkappa,y,r),\mathfrak{I}_{0^{+}}^{\varphi,\rho}\overline{\upsilon}(\varkappa,y,r)$ are defined as (\ref{i4}).
	\end{definition}
\begin{definition}\label{2*}
	Let $\upsilon:\boldsymbol{\varOmega} \rightarrow \hat{E_{f}}$, $\varphi=(\varphi_{1},\varphi_{2})\in(0,1)^{2}$ and $\rho=(\rho_{1},\rho_{2})>0.$. The mixed fuzzy Katugampola $g\mathcal{H}$-fractional derivative of order $\varphi=(\varphi_{1},\varphi_{2})$ is defined by
	\[\mathfrak{D}_{0^+}^{\varphi,\rho}\upsilon(\varkappa,y) =\varkappa^{1-\rho_{1}}y^{1-\rho_{2}}\mathfrak{D}^{2}_{\varkappa,y}( \mathfrak{I}_{0^{+}}^{1-\varphi,\rho}\upsilon (\varkappa,y)),\]
	if the g$\mathcal{H}$-derivative $\mathfrak{D}^{2}_{\varkappa,y}(	\mathfrak{I}_{0^{+}}^{1-\varphi,\rho}\upsilon (\varkappa,y))$ exists for $(\varkappa,y)\in \boldsymbol{\varOmega}$, where
	\begin{equation*}
	\begin{aligned}
	\mathfrak{I}_{0^{+}}^{1-\varphi,\rho}\upsilon (\varkappa,y)& = \frac{\rho_{1}^{\varphi_{1}}\rho_{2}^{\varphi
			_2}}{\Gamma(1-\varphi_{1})\Gamma(1-\varphi_{2})} \int_{0}^{\varkappa}\int_{0}^{y}\frac{s^{\rho_{1}-1}t^{\rho_{2}-1}}{(\varkappa^{\rho_{1}}-s^{\rho_{1}})^{\varphi_{1}}(y^{\rho_{2}}-t^{\rho_{2}})^{\varphi_{2}}} \mu(s,t)dtds.
		\end{aligned}
		\end{equation*}
	\end{definition}
	\begin{remark}\label{i15}
		Let $\upsilon \in AC(\boldsymbol{\varOmega},\hat{E_{f}})$. Then 
	\begin{itemize}
		\item If $\mathfrak{D}^{2}_{\varkappa,y}$ is a $(\hat{i})$-g$\mathcal{H}$ derivative then $\upsilon$ is (1)-Katugampola g$\mathcal{H}$ differentiable and 
		\[[\mathfrak{D}_{0^+}^{\varphi,\rho}\upsilon(\varkappa,y)]^{r} = [\mathfrak{D}_{0^+}^{\varphi,\rho}\underline{\upsilon}(\varkappa,y,r),\mathfrak{D}_{0^+}^{\varphi,\rho}\overline{\upsilon}(\varkappa,y,r)].\]
		\item  If $\mathfrak{D}^{2}_{\varkappa,y}$ is a $(\widehat{ii})$-g$\mathcal{H}$ derivative then $\upsilon$ is   (2)-Katugampola g$\mathcal{H}$ differentiable and 
		\[[\mathfrak{D}_{0^+}^{\varphi,\rho}\upsilon(\varkappa,y)]^{r} = [\mathfrak{D}_{0^+}^{\varphi,\rho}\overline{\upsilon}(\varkappa,y,r),\mathfrak{D}_{0^+}^{\varphi,\rho}\underline{\upsilon}(\varkappa,y,r)].\]
	\end{itemize}
\end{remark}
\begin{proposition}
	Let $\varphi=(\varphi_{1},\varphi_{2}),\psi=(\psi_{1},\psi_{2})\in(0,1)^{2},$ $\upsilon\in C(\boldsymbol{\varOmega},\hat{E_{f}}) \cap L(\boldsymbol{\varOmega},\hat{E_{f}})$ and $\rho=(\rho_{1},\rho_{2})>0.$ Then
	\begin{equation}\label{i6*}
			\mathfrak{I}_{0^{+}}^{\varphi,\rho}\mathfrak{I}_{0^{+}}^{\psi,\rho}\upsilon(\varkappa,y)=\mathfrak{I}_{0^{+}}^{\varphi+\psi,\rho}\upsilon (\varkappa,y),
	\end{equation}
\text{and}
	\begin{equation} \label{i6}
	\mathfrak{D}_{0^+}^{\varphi,\rho}\mathfrak{I}_{0^{+}}^{\varphi,\rho}\upsilon (\varkappa,y)=\upsilon (\varkappa,y).
\end{equation}
\end{proposition}
\begin{proof}
	For $(\varkappa,y) \in \boldsymbol{\varOmega}$ we have 
	\begin{equation*}
		\begin{aligned}
			\mathfrak{I}_{0^{+}}^{\varphi,\rho}\mathfrak{I}_{0^{+}}^{\psi,\rho}\upsilon(\varkappa,y)&=\frac{\rho_{1}^{1-\varphi_{1}}\rho_{2}^{1-\varphi
					_2}}{\Gamma(\varphi_{1})\Gamma(\varphi_{2})} \int_{0}^{\varkappa}\int_{0}^{y}\frac{s^{\rho_{1}-1}t^{\rho_{2}-1}\mathfrak{I}_{0^{+}}^{\psi,\rho}\upsilon(\varkappa,y)}{(\varkappa^{\rho_{1}}-s^{\rho_{1}})^{1-\varphi_{1}}(y^{\rho_{2}}-t^{\rho_{2}})^{1-\varphi_{2}}}dtds. 
				\end{aligned}
			\end{equation*}
				Now by using Fuzzy Fubini’s Theorem (see \cite{r33}) we have
				   	\begin{equation}\label{i7}
					\begin{aligned}
		\mathfrak{I}_{0^{+}}^{\varphi,\rho}\mathfrak{I}_{0^{+}}^{\psi,\rho}\upsilon(\varkappa,y)=&\frac{\rho_{1}^{1-\psi_{1}}\rho_{2}^{1-\psi
				_2}\rho_{1}^{1-\varphi_{1}}\rho_{2}^{1-\varphi
								_2}}{\Gamma(\psi_{1})\Gamma(\psi_{2})\Gamma(\varphi_{1})\Gamma(\varphi_{2})} \int_{0}^{\varkappa}\int_{0}^{y}\frac{s_{1}^{\rho_{1}-1}s_{2}^{\rho_{2}-1}}{(\varkappa^{\rho_{1}}-_{1}^{\rho_{1}})^{1-\varphi_{1}}(y^{\rho_{2}}-s_{2}^{\rho_{2}})^{1-\varphi_{2}}}\times\\
							&\int_{0}^{s_{1}}\int_{0}^{s_{2}}\frac{t_{1}^{\rho_{1}-1}t_{2}^{\rho_{2}-1}}{(s_{1}^{\rho_{1}}-t_{1}^{\rho_{1}})^{1-\psi_{1}}(s_{2}^{\rho_{2}}-t_{2}^{\rho_{2}})^{1-\psi_{2}}}\upsilon(t_{1},t_{2})dt_{2}dt_{1}ds_{2}ds_{1}
							\\=&\frac{\rho_{1}^{1-\psi_{1}}\rho_{2}^{1-\psi
									_2}\rho_{1}^{1-\varphi_{1}}\rho_{2}^{1-\varphi
									_2}}{\Gamma(\psi_{1})\Gamma(\psi_{2})\Gamma(\varphi_{1})\Gamma(\varphi_{2})} \int_{0}^{\varkappa}\int_{0}^{y} t_{1}^{\rho_{1}-1}t_{2}^{\rho_{2}-1}\upsilon(t_{1},t_{2})\times\\ 
							&\int_{t_{1}}^{\varkappa}\int_{t_{2}}^{y}\frac{s_{1}^{\rho_{1}-1}s_{2}^{\rho_{2}-1}ds_{2}ds_{1}dt_{2}dt_{1}}{(\varkappa^{\rho_{1}}-s_{1}^{\rho_{1}})^{1-\varphi_{1}}(y^{\rho_{2}}-s_{2}^{\rho_{2}})^{1-\varphi_{2}}(s_{1}^{\rho_{1}}-t_{1}^{\rho_{1}})^{1-\psi_{1}}(s_{2}^{\rho_{2}}-t_{2}^{\rho_{2}})^{1-\psi_{2}}}.
			\end{aligned}
	\end{equation}	
Now we set
\[A=\int_{t_{1}}^{\varkappa}\int_{t_{2}}^{y}\frac{s_{1}^{\rho_{1}-1}s_{2}^{\rho_{2}-1}ds_{2}ds_{1}}{(\varkappa^{\rho_{1}}-s_{1}^{\rho_{1}})^{1-\varphi_{1}}(y^{\rho_{2}}-s_{2}^{\rho_{2}})^{1-\varphi_{2}}(s_{1}^{\rho_{1}}-t_{1}^{\rho_{1}})^{1-\psi_{1}}(s_{2}^{\rho_{2}}-t_{2}^{\rho_{2}})^{1-\psi_{2}}}.\] Let
 \[X=\frac{(s_{1}^{\rho_{1}}-t_{1}^{\rho_{1}})}{(\varkappa^{\rho_{1}}-s_{1}^{\rho_{1}})} ~ \text{and}~	Y=\frac{(s_{2}^{\rho_{2}}-t_{2}^{\rho_{2}})}{(y^{\rho_{2}}-s_{2}^{\rho_{2}})}.\]
 By substituting X and Y in $A$ we get
 \begin{equation}\label{i8}
 	\begin{aligned}
 		A&=\int_{t_{1}}^{\varkappa}\frac{s_{1}^{\rho_{1}-1}}{(\varkappa^{\rho_{1}}-s_{1}^{\rho_{1}})^{1-\varphi_{1}}(s_{1}^{\rho_{1}}-t_{1}^{\rho_{1}})^{1-\psi_{1}}}ds_{1}\times\int_{t_{2}}^{y}\frac{s_{2}^{\rho_{2}-1}}{(y^{\rho_{2}}-s_{2}^{\rho_{2}})^{1-\varphi_{2}}(s_{2}^{\rho_{2}}-t_{2}^{\rho_{2}})^{1-\psi_{2}}}ds_{1}\\\\&=
 		\frac{(\varkappa^{\rho_{1}}-t^{\rho_{1}})^{\varphi_{1}+\psi_{1}-1}(y^{\rho_{2}}-t^{\rho_{2}})^{\varphi_{2}+\psi_{2}-1}}{\rho_{1}\rho_{2}}\int_{0}^{1}(1-X)^{\varphi_{1}-1}X^{\psi_{1}}dX\int_{0}^{1}(1-Y)^{\varphi_{2}-1}Y^{\psi_{2}}dY\\\\&=
 		\frac{(\varkappa^{\rho_{1}}-t^{\rho_{1}})^{\varphi_{1}+\psi_{1}-1}(y^{\rho_{2}}-t^{\rho_{2}})^{\varphi_{2}+\psi_{2}-1}}{\rho_{1}\rho_{2}}\frac{\Gamma(\varphi_{1})\Gamma(\psi_{1})}{\Gamma(\varphi_{1}+\psi_{1})}\frac{\Gamma(\varphi_{2})\Gamma(\psi_{2})}{\Gamma(\varphi_{2}+\psi_{2})}.
 	\end{aligned}
\end{equation}
The result of (\ref{i6*}) is obtained by merging equations (\ref{i7}) to (\ref{i8}) as follows.
\begin{equation*}
	\begin{aligned}
	\mathfrak{I}_{0^{+}}^{\varphi,\rho}\mathfrak{I}_{0^{+}}^{\psi,\rho}\upsilon(\varkappa,y)=&\frac{\rho_{1}^{1-(\varphi_{1}+\psi_{1})}\rho_{2}^{1-(\varphi_{2}+\psi
		_2)}}{\Gamma(\varphi_{1}+\psi_{1})\Gamma(\varphi_{2}+\psi_{2})} \int_{0}^{\varkappa}\int_{0}^{y} \frac{t_{1}^{\rho_{1}-1}t_{2}^{\rho_{2}-1}}{(\varkappa^{\rho_{1}}-t^{\rho_{1}})^{1-(\varphi_{1}+\psi_{1})}(y^{\rho_{2}}-t^{\rho_{2}})^{1-(\varphi_{2}+\psi_{2})}}\times\\&
	\upsilon(t_{1},t_{2})dt_{2}dt_{1}\\=&\mathfrak{I}_{0^{+}}^{\varphi+\psi,\rho}\upsilon (\varkappa,y).
		\end{aligned}
\end{equation*}
The proof of (\ref{i6}) is similar to the proof Lemma (2.3) in \cite{r11}.
\end{proof}
\begin{definition} \label{i14}
		Let $\upsilon:\boldsymbol{\varOmega} \rightarrow \hat{E_{f}}$ be a fuzzy valued function, $\varphi=(\varphi_{1},\varphi_{2})\in(0,1)^{2}$ and $\rho=(\rho_{1},\rho_{2})>0.$ The mixed fuzzy $\mathcal{CK}$ g$\mathcal{H}$-fractional derivative of order $\varphi=(\varphi_{1},\varphi_{2})\in(0,1)^{2}$ is defined by
			\[^{C}\mathfrak{D}_{0^+}^{\varphi,\rho}\upsilon(\varkappa,y)=\mathfrak{D}_{0^+}^{\varphi,\rho}\big(\upsilon(\varkappa,y)\ominus\upsilon(\varkappa,0)\ominus\upsilon(0,y)+\upsilon(0,0) \big),\]
			such that $\mathfrak{D}_{0^+}^{\varphi,\rho}$ exists on $\boldsymbol{\varOmega}.$
\end{definition}
\begin{lemma}\label{i6**}
	If $\upsilon \in AC(\boldsymbol{\varOmega},\hat{E_{f}})$ is a d-monotone fuzzy function. Then the mixed fuzzy $\mathcal{CK}$ g$\mathcal{H}$-fractional derivative of order $\varphi=(\varphi_{1},\varphi_{2})\in(0,1)^{2}$ is given by
		\begin{equation*}
		\begin{aligned}
		^{C}\mathfrak{D}_{0^+}^{\varphi,\rho}\upsilon(\varkappa,y)&=\mathfrak{I}_{0^{+}}^{1-\varphi,\rho}( \varkappa^{1-\rho_{1}}y^{1-\rho_{2}}\mathfrak{D}^{2}_{\varkappa,y} \upsilon (\varkappa,y))\\& = \frac{\rho_{1}^{\varphi_{1}}\rho_{2}^{\varphi
					_2}}{\Gamma(1-\varphi_{1})\Gamma(1-\varphi_{2})} \int_{0}^{\varkappa}\int_{0}^{y}\frac{\mathfrak{D}^{2}_{\varkappa,y} \upsilon (\varkappa,y)}{(\varkappa^{\rho_{1}}-s^{\rho_{1}})^{\varphi_{1}}(y^{\rho_{2}}-t^{\rho_{2}})^{\varphi_{2}}}dtds.
		\end{aligned}
	\end{equation*}
\begin{proof}
Since $\upsilon \in AC(\boldsymbol{\varOmega},\hat{E_{f}})$ Therefore $\mathfrak{D}_{0^+}^{\varphi,\rho}\upsilon(\varkappa,y)$ exists for $(\varkappa,y)\in \boldsymbol{\varOmega}.$
Now we show the proof in two cases:\\
First if $\upsilon$ is (1)-Katugampola g$\mathcal{H}$ differentiable  then by Remark (2.6-(i)) in \cite{r27} we have
\begin{equation}\label{i9}
\begin{split}
^{C}\mathfrak{D}_{0^+}^{\varphi,\rho}\upsilon(\varkappa,y)=&\mathfrak{D}_{0^+}^{\varphi,\rho}\big(\upsilon(\varkappa,y)\ominus\upsilon(\varkappa,0)\ominus\upsilon(0,y)+\upsilon(0,0) \big)\\=&\mathfrak{D}_{0^+}^{\varphi,\rho}\big(\upsilon(\varkappa,y)\ominus\upsilon(\varkappa,0)\big)\ominus\mathfrak{D}_{0^+}^{\varphi,\rho}\big(\upsilon(0,y)\ominus \upsilon(0,0)\big)
\\ =&\mathfrak{D}_{0^+}^{\varphi,\rho}\upsilon(\varkappa,y)\ominus\mathfrak{D}_{0^+}^{\varphi,\rho}\upsilon(\varkappa,0)\ominus\mathfrak{D}_{0^+}^{\varphi,\rho}\upsilon(0,y)\\&+ \frac{\varkappa^{-\varphi_{1}}y^{-\varphi_{2}}\rho_{1}^{\varphi_{1}} \rho_{2}^{\varphi_{2}}}{\Gamma(1-\varphi_{1})\Gamma(1-\varphi_{2})}\upsilon(0,0).
\end{split}
\end{equation}
By Remark (\ref{i15}) and (\ref{i5}) for $r\in [0,1]$ we have
\begin{equation}
	\begin{split}
[\mathfrak{D}_{0^+}^{\varphi,\rho}\upsilon(\varkappa,y)]^{r} =& [\mathfrak{D}_{0^+}^{\varphi,\rho}\underline{\upsilon}(\varkappa,y,r),\mathfrak{D}_{0^+}^{\varphi,\rho}\overline{\upsilon}(\varkappa,y,r)]\\=&[^{C}\mathfrak{D}_{0^+}^{\varphi,\rho}\underline{\upsilon}(\varkappa,y,r),^{C}\mathfrak{D}_{0^+}^{\varphi,\rho}\overline{\upsilon}(\varkappa,y,r)]\\&+[\mathfrak{D}_{0^+}^{\varphi,\rho}\underline{\upsilon}(\varkappa,0,r),\mathfrak{D}_{0^+}^{\varphi,\rho}\overline{\upsilon}(0,y,r)]
\\&+[\mathfrak{D}_{0^+}^{\varphi,\rho}\underline{\upsilon}(\varkappa,0,r),\mathfrak{D}_{0^+}^{\varphi,\rho}\overline{\upsilon}(0,y,r)]
\\&-\frac{\varkappa^{-\varphi_{1}}y^{-\varphi_{2}}\rho_{1}^{\varphi_{1}} \rho_{2}^{\varphi_{2}}}{\Gamma(1-\varphi_{1})\Gamma(1-\varphi_{2})}[\underline{\upsilon}(0,0,r),\overline{\upsilon}(0,0,r)].
\end{split}
\end{equation}
 For every $(\varkappa,y) \in \boldsymbol{\varOmega}$ and $r \in [0,1]$ we get
 \begin{equation}\label{i10}
 	\begin{split}
 		\mathfrak{D}_{0^+}^{\varphi,\rho}\upsilon(\varkappa,y) =&
 		\mathfrak{I}_{0^{+}}^{1-\varphi,\rho}( \varkappa^{1-\rho_{1}}y^{1-\rho_{2}}\mathfrak{D}^{2}_{\varkappa,y} \upsilon (\varkappa,y))+\mathfrak{D}_{0^+}^{\varphi,\rho}{\upsilon}(\varkappa,0)+\mathfrak{D}_{0^+}^{\varphi,\rho}{\upsilon}(0,y)\\& \ominus\frac{\varkappa^{-\varphi_{1}}y^{-\varphi_{2}}\rho_{1}^{\varphi_{1}} \rho_{2}^{\varphi_{2}}}{\Gamma(1-\varphi_{1})\Gamma(1-\varphi_{2})}{\upsilon}(0,0).
 	\end{split}
 \end{equation}
Substituting(\ref{i10}) in (\ref{i9}) we get the result.

Next if  $\upsilon$ is (2)-Katugampola g$\mathcal{H}$ differentiable then by Remark (2.6-(ii)) in \cite{r27}  we have
\begin{equation}\label{i11}
	\begin{split}
		^{C}\mathfrak{D}_{0^+}^{\varphi,\rho}\upsilon(\varkappa,y)=&\mathfrak{D}_{0^+}^{\varphi,\rho}\big(\upsilon(\varkappa,y)\ominus\upsilon(\varkappa,0)\ominus\upsilon(0,y)+\upsilon(0,0) \big)\\=&\mathfrak{D}_{0^+}^{\varphi,\rho}\big(\upsilon(\varkappa,y)\ominus\upsilon(\varkappa,0)\big)+(-1)\mathfrak{D}_{0^+}^{\varphi,\rho}\big(\upsilon(0,y)\ominus \upsilon(0,0)\big)
		\\ =&\mathfrak{D}_{0^+}^{\varphi,\rho}\upsilon(\varkappa,y)+(-1)\mathfrak{D}_{0^+}^{\varphi,\rho}\upsilon(\varkappa,0)+(-1)\mathfrak{D}_{0^+}^{\varphi,\rho}\upsilon(0,y)\\&+ \frac{\varkappa^{-\varphi_{1}}y^{-\varphi_{2}}\rho_{1}^{\varphi_{1}} \rho_{2}^{\varphi_{2}}}{\Gamma(1-\varphi_{1})\Gamma(1-\varphi_{2})}\upsilon(0,0).
	\end{split}
\end{equation}
By Remark (\ref{i15}) and (\ref{i5}) for $r\in [0,1]$ we have
\begin{equation}
	\begin{split}
		[^{C}\mathfrak{D}_{0^+}^{\varphi,\rho}\upsilon(\varkappa,y)]^{r} =&	[\mathfrak{I}_{0^{+}}^{1-\varphi,\rho}( \varkappa^{1-\rho_{1}}y^{1-\rho_{2}}\mathfrak{D}^{2}_{\varkappa,y} \overline{\upsilon} (\varkappa,y)),\mathfrak{I}_{0^{+}}^{1-\varphi,\rho}( \varkappa^{1-\rho_{1}}y^{1-\rho_{2}}\mathfrak{D}^{2}_{\varkappa,y} \underline{\upsilon} (\varkappa,y))]\\=& [\mathfrak{D}_{0^+}^{\varphi,\rho}\overline{\upsilon}(\varkappa,y,r),\mathfrak{D}_{0^+}^{\varphi,\rho}\underline{\upsilon}(\varkappa,y,r)]\\&+(-1)[\mathfrak{D}_{0^+}^{\varphi,\rho}\underline{\upsilon}(\varkappa,0,r),\mathfrak{D}_{0^+}^{\varphi,\rho}\overline{\upsilon}(0,y,r)]
		\\&+(-1)[\mathfrak{D}_{0^+}^{\varphi,\rho}\underline{\upsilon}(\varkappa,0,r),\mathfrak{D}_{0^+}^{\varphi,\rho}\overline{\upsilon}(0,y,r)]
		\\&+\frac{\varkappa^{-\varphi_{1}}y^{-\varphi_{2}}\rho_{1}^{\varphi_{1}} \rho_{2}^{\varphi_{2}}}{\Gamma(1-\varphi_{1})\Gamma(1-\varphi_{2})}[\overline{\upsilon}(0,0,r),\underline{\upsilon}(0,0,r)].
	\end{split}
\end{equation}
For every $(\varkappa,y) \in \boldsymbol{\varOmega}$ and $r \in [0,1]$ we get
\begin{equation}\label{i12}
	\begin{split}
		\mathfrak{I}_{0^{+}}^{1-\varphi,\rho}( \varkappa^{1-\rho_{1}}y^{1-\rho_{2}}\mathfrak{D}^{2}_{\varkappa,y} \upsilon (\varkappa,y)) =&\mathfrak{D}_{0^+}^{\varphi,\rho}\upsilon(\varkappa,y)
		+\mathfrak{D}_{0^+}^{\varphi,\rho}(-1)\upsilon(\varkappa,0)+\mathfrak{D}_{0^+}^{\varphi,\rho}(-1)\upsilon(0,y)\\& +\frac{\varkappa^{-\varphi_{1}}y^{-\varphi_{2}}\rho_{1}^{\varphi_{1}} \rho_{2}^{\varphi_{2}}}{\Gamma(1-\varphi_{1})\Gamma(1-\varphi_{2})}{\upsilon}(0,0).
	\end{split}
\end{equation}
Substituting (\ref{i12}) in (\ref{i11}) we get the result.\\
The proof is complete.
\end{proof}
\end{lemma}
\text{Note:} 
\begin{itemize}
	\item If  $\mathfrak{D}^{2}_{\varkappa,y}$is a $(\hat{i})$-g$\mathcal{H}$ derivative then $\upsilon$ is (1)-Caputo Katugampola g$\mathcal{H}$ differentiable.
	\item  If  $\mathfrak{D}^{2}_{\varkappa,y}$is a $(\widehat{ii})$-g$\mathcal{H}$ derivative then $\upsilon$ is   (2)-Caputo Katugampola g$\mathcal{H}$ differentiable.
\end{itemize}
\begin{proposition}\label{i13}
	Let $\upsilon\in AC(\boldsymbol{\varOmega},\hat{E_{f}})$ be an integrable function and $\varphi=(\varphi_{1}\varphi_{2})\in(0,1)^{2}$, then
	\begin{enumerate}
		\item if $\upsilon$ is (1)-Caputo Katugampola g$\mathcal{H}$ differentiable we have
			\begin{equation}
			{\mathfrak{I}_{0^{+}}^{\varphi,\rho}}~ {^{C}\mathfrak{D}_{0^+}^{\varphi,\rho}}\upsilon(\varkappa,y)=\upsilon(\varkappa,y)\ominus\upsilon(\varkappa,0)\ominus\upsilon(0,y)+\upsilon(0,0), 
		\end{equation}
		\item  if $\upsilon$ is (2)-Caputo Katugampola g$\mathcal{H}$ differentiable we have
		\begin{equation}
			{\mathfrak{I}_{0^{+}}^{\varphi,\rho}}~ {^{C}\mathfrak{D}_{0^+}^{\varphi,\rho}}\upsilon(\varkappa,y)=\ominus(-1)\big[\upsilon(\varkappa,y)\ominus\upsilon(\varkappa,0)\ominus\upsilon(0,y)+\upsilon(0,0)\big]. 
		\end{equation}
\end{enumerate}
\end{proposition}
\begin{proof}
	 By using (\ref{i6*}) and Lemma (\ref{i6**}) we have
	 \begin{equation}
	 	\begin{split}
	 	{\mathfrak{I}_{0^{+}}^{\varphi,\rho}}~ {^{C}\mathfrak{D}_{0^+}^{\varphi,\rho}}\upsilon(\varkappa,y)&={\mathfrak{I}_{0^{+}}^{\varphi,\rho}}[\mathfrak{I}_{0^{+}}^{1-\varphi,\rho}( \varkappa^{1-\rho_{1}}y^{1-\rho_{2}}\mathfrak{D}^{2}_{\varkappa,y} \upsilon (\varkappa,y))]\\&={\mathfrak{I}_{0^{+}}^{1,\rho}}(\varkappa^{1-\rho_{1}}y^{1-\rho_{2}}\mathfrak{D}^{2}_{\varkappa,y} \upsilon (\varkappa,y))\\&=\int_{0}^{\varkappa}\int_{0}^{y}\frac{\partial^{2}}{\partial s \partial t} \upsilon (s,t))dtds.
	 		\end{split}
	 \end{equation}
 Now, if $\upsilon$ is (1)-Caputo Katugampola g$\mathcal{H}$ differentiable that is $\upsilon$ and $\frac{\partial \upsilon}{\partial s}$ are  $(i)$-g$\mathcal{H}$ derivative $\boldsymbol{\varOmega}$ or $(ii)$-g$\mathcal{H}$ derivative.
  If $\upsilon$ and $\frac{\partial \upsilon}{\partial s}$ are  $(i)$-g$\mathcal{H}$ derivative we have
 \begin{equation}
 	\begin{split}
 	\int_{0}^{\varkappa}\int_{0}^{y}\frac{\partial^{2}}{\partial s \partial t} \upsilon (s,t))dtds&= \int_{o}^{x}\big(\frac{\partial }{\partial s}\upsilon (s,y)\ominus\frac{\partial }{\partial s}\upsilon (s,0)\big)ds\\&=\int_{o}^{x}\frac{\partial }{\partial s}\upsilon (s,y)ds \ominus\int_{o}^{x}\frac{\partial }{\partial s}\upsilon (s,0)ds.
 \end{split}
 \end{equation}
Therefore 
\[	\int_{0}^{\varkappa}\int_{0}^{y}\frac{\partial^{2}}{\partial s \partial t} \upsilon (s,t))dtds=\upsilon(\varkappa,y)\ominus\upsilon(\varkappa,0)\ominus\upsilon(0,y)+\upsilon(0,0).\]
Similarly we get the same result for first item if  $\upsilon$ and $\frac{\partial \upsilon}{\partial s}$ are  $(ii)$-g$\mathcal{H}$ derivative. \\
Now, if $\upsilon$ is (2)-Caputo Katugampola g$\mathcal{H}$ differentiable that is $\upsilon$ is $(i)$-g$\mathcal{H}$ derivative and $\frac{\partial \upsilon}{\partial s}$ is $(ii)$-g$\mathcal{H}$ derivative or $\upsilon$ is $(ii)$-g$\mathcal{H}$ derivative and $\frac{\partial \upsilon}{\partial s}$ is $(i)$-g$\mathcal{H}$ derivative.\\
If $\upsilon$ is $(ii)$-g$\mathcal{H}$ derivative and $\frac{\partial \upsilon}{\partial s}$ is $(i)$-g$\mathcal{H}$ derivative we get
 \begin{equation}
 	\begin{split}
 		\int_{0}^{\varkappa}\int_{0}^{y}\frac{\partial^{2}}{\partial s \partial t} \upsilon (s,t))dtds&= \int_{o}^{x}\big(\frac{\partial }{\partial s}\upsilon (s,y)\ominus\frac{\partial }{\partial s}\upsilon (s,0)\big)ds\\&=\int_{o}^{x}\frac{\partial }{\partial s}\upsilon (s,y)ds \ominus\int_{o}^{x}\frac{\partial }{\partial s}\upsilon (s,0)ds\\&=(-1)\upsilon (0,y)\ominus(-1)\upsilon (x,y)\ominus\big((-1)\upsilon (0,0)\ominus(-1)\upsilon (x,0)\big).
 	\end{split}
 \end{equation}
Therefore 
\[\int_{0}^{\varkappa}\int_{0}^{y}\frac{\partial^{2}}{\partial s \partial t} \upsilon (s,t))dtds=\ominus(-1)\big[\upsilon(\varkappa,y)\ominus\upsilon(\varkappa,0)\ominus\upsilon(0,y)+\upsilon(0,0)\big].\]
Similarly we can get the same result for second item if $\upsilon$ is $(ii)$-g$\mathcal{H}$ derivative and $\frac{\partial \upsilon}{\partial s}$ is $(i)$-g$\mathcal{H}$ derivative.
\end{proof}
	\section{ Darboux problem for Fuzzy Fractional PDE with Caputo  Katogumpola g$\mathcal{H}$ derivative }
	In this section, we discuss the existence and uniqueness of a solution to the system (\ref{i1}).
	
	Now we give some concepts and theorems which are needed to solve (\ref{i1}) and (\ref{i2}).
\begin{definition}\label{s00}\cite{r22,r26}
   Let $B \subseteq \hat{E_{f}}$. For every $z\in B$, a subset B is called compact-supported if there exists a compact set $L \subseteq \mathbb{R}$ such that $[z]^{0}\subseteq L$.
\end{definition}
\begin{definition}\cite{r22,r26}
	A subset $B \subseteq \hat{E_{f}}$ is called level-equicontinuous at $r_{0}\in [0, 1]$, if for all $\varepsilon>0$
	there exists $\delta > 0$ such that  $|r-r_{0}|<\delta$ $\Rightarrow$ $d_{H}([u]^{r}, [u]^{r_{0}} ) <\varepsilon$ for each $u \in B$. 
\end{definition}
	\begin{theorem}\label{s1*}\cite{r22,r26}
		 Suppose B is a subset of $\hat{E_{f}}$ with compact support. Then the following statements are equivalent
		 \begin{itemize}
		 	\item B is a relatively compact subset of $(\hat{E_{f}}, \mathbb{D})$.
		 	\item B is level-equicontinuous on $[0, 1].$
		 \end{itemize}
	\end{theorem}
\begin{theorem}\cite{r22}
	Let Z be a nonempty, closed, bounded and convex subset of a Banach space $\mathcal{A}$ , and suppose that $N:Z \rightarrow Z$ is a relatively compact operator in $\mathcal{A}$. Then N has at least one fixed point in Z.
\end{theorem}
\begin{theorem}\cite{r34}
	Let $(\mathcal{A},\mathbb{D})$ be
	a complete metric space, then each contraction mapping
	 $N:\mathcal{A} \rightarrow \mathcal{A}$ has a unique fixed point x of $N$ in $\mathcal{A}$.
\end{theorem}
	    For $(\varkappa,y)\in \boldsymbol{\varOmega}$ and  $\xi_{1}$, $\xi_{2}$ are known functions we define the function $h$ by 
	\[h(\varkappa,y)=\xi_{1}(\varkappa)+\xi_{2}(y)-\xi_{1}(0).\]
	
	The following lemma gives the equivalent of integral equations for (\ref{i1}). 
\begin{lemma}\label{s0}
Let $\upsilon\in C(\boldsymbol{\varOmega},\hat{E_{f}})$ be a fuzzy valued function satisfying (\ref{i1}), $\mathcal{F}: \boldsymbol{\varOmega} \times \hat{E_{f}} \rightarrow \hat{E_{f}}$ and $\varphi=(\varphi_{1},\varphi_{2})$ where $0<\varphi_{1},\varphi_{2}<1$ and $\rho=(\rho_{1},\rho_{2})>0$, then one of the following integral equations  is a solution of (\ref{i1}):
\begin{equation}\label{s1}
	\upsilon(\varkappa,y)=h(\varkappa,y)+ \mathfrak{I}^{\varphi,\rho}_{0^{+}}\big[\mathcal{F}(\varkappa,y,\upsilon(\varkappa,y))\big], 
\end{equation}
or
\begin{equation}\label{s2}
	\upsilon(\varkappa,y)=h(\varkappa,y)\ominus(-1) \mathfrak{I}^{\varphi,\rho}_{0^{+}}\big[\mathcal{F}(\varkappa,y,\upsilon(\varkappa,y))\big]. 
\end{equation}
\end{lemma}
\begin{proof}
	Let $\upsilon(\varkappa,y)$ be a solution of equation (\ref{i1}). Applying the mixed Katugampola fuzzy fractional integral 
	($\mathfrak{I}_{0^{+}}^{\varphi,\rho}$) to both sides of equation (\ref{i1}), we get
	\[	{\mathfrak{I}_{0^{+}}^{\varphi,\rho}}~ {^{C}\mathfrak{D}_{0^+}^{\varphi,\rho}}\upsilon(\varkappa,y)=\mathfrak{I}^{\varphi,\rho}_{0^{+}}\mathcal{F}(\varkappa,y,\upsilon(\varkappa,y)), \] 
	and from the  Proposition (\ref{i13}) we get (\ref{s1}) and (\ref{s2}).
\end{proof}
In this paper we  study our results for solution (\ref{s1}) while the outcomes corresponding to (\ref{s2}) can be demonstrated in a similar manner.

	Now, in our next theorem we prove the existence of the solution (\ref{s1}) for the problem (\ref{i1}) by employing Schauder's fixed point theorem. Let $k>0$ be a positive constant, we define $\mathcal{B}(0,k)=\{\upsilon\in \hat{E_{f}} : \mathbb{D}[\upsilon,0] \leq k \}.$ 
	\begin{theorem} \label{s3}
		Let $k>0$ and $\varkappa_{1}, y_{1}>0$ such that
		$\varkappa_{1}^{\rho_{1}}y_{1}^{\rho_{2}}\leq\big[\frac{k\Gamma(1+\varphi_{1})\Gamma(1+\varphi_{2})}{2M}\rho_{1}^{\varphi_{1}}\rho_{2}^{\varphi_{2}}\big]^\frac{1}{\varphi_{1}\varphi_{2}}$, $h\in C[{\boldsymbol{\varOmega}}, \mathcal{B}(0,\frac{k}{2})]$  and Let $S=min\{\varkappa_{1},a\}$, $T=min\{y_{1},b\}$.
		Define 
		\[M=\sup_{(\varkappa,y,\upsilon) \in {\boldsymbol{\varOmega}}, \times \mathcal{B}(0,k)}H^{*}[\mathcal{F}(\varkappa,y,\upsilon(\varkappa,y)),0]  ,\]
		where $M$ is a positive constant. Then for all $(\varkappa,y)\in\widetilde{\boldsymbol{\varOmega}}=[0,S]\times[0,T]$  there exists at least a function $\upsilon \in C[\widetilde{\boldsymbol{\varOmega}},\hat{E_{f}}]$ that solve the problem (\ref{i1}).
	\end{theorem}
	\begin{proof}
		We define the operator $\mathbb{A}$ by
		\begin{equation*}
			\mathbb{A}	\upsilon(\varkappa,y)=h(\varkappa,y)+ \mathfrak{I}^{\varphi,\rho}_{0^{+}}\big[\mathcal{F}(\varkappa,y,\upsilon(\varkappa,y))\big].	
		\end{equation*}
	Now we show that the operator $\mathbb{A}$ has a fixed point by the following steps\\ 
	Step 1. $\mathbb{A}$ is continuous. For $(\varkappa,y) \in \widetilde{\boldsymbol{\varOmega}}$ we define $\widetilde{W}=\{\upsilon\in C(\widetilde{\boldsymbol{\varOmega}},\hat{E_{f}}) : \mathbb{D}(\upsilon,h) \leq k \}$.\\
	Firstly, we show that $\mathbb{A}$ maps the set $\widetilde{W}$ in to itself, for every $\upsilon \in \widetilde{W}$ and $(\varkappa,y) \in \widetilde{\boldsymbol{\varOmega}}$ we have
	\begin{equation*}
		\begin{split}
			{\mathbb{D}}(\mathbb{A}\upsilon(\varkappa,y),0) \leq &{\mathbb{D}}(h(\varkappa,y),0)+ \frac{\rho_{1}^{1-\varphi_{1}}\rho_{2}^{1-\varphi
					_2}}{\Gamma(\varphi_{1})\Gamma(\varphi_{2})} \int_{0}^{\varkappa}\int_{0}^{y}\frac{s^{\rho_{1}-1}t^{\rho_{2}-1}}{(\varkappa^{\rho_{1}}-s^{\rho_{1}})^{1-\varphi_{1}}(y^{\rho_{2}}-t^{\rho_{2}})^{1-\varphi_{2}}}\times\\&  {\mathbb{D}}\big(\mathcal{F}(s,t,\upsilon(s,t),0)\big)dtds\\\\\leq&{\mathbb{D}}(h(\varkappa,y),0)+\frac{\rho_{1}^{1-\varphi_{1}}\rho_{2}^{1-\varphi
					_2}M}{\Gamma(\varphi_{1})\Gamma(\varphi_{2})} \int_{0}^{\varkappa}\int_{0}^{y}\frac{s^{\rho_{1}-1}t^{\rho_{2}-1}}{(\varkappa^{\rho_{1}}-s^{\rho_{1}})^{1-\varphi_{1}}(y^{\rho_{2}}-t^{\rho_{2}})^{1-\varphi_{2}}}dtds\\\leq&{\mathbb{D}}(h(\varkappa,y),0)+\frac{M\varkappa^{\rho_{1}\varphi_{1}}y^{\rho_{1}\varphi_{2}}}{\rho_{1}^{\varphi_{1}}\rho_{2}^{\varphi
					_2}\Gamma(\varphi_{1}+1)\Gamma(\varphi_{2}+1)}\\\leq&k.
		\end{split}
	\end{equation*}
	Hence $\mathbb{A} \in \widetilde{W}.$ 
	
	 Now we show that $\mathbb{A}$ is continuous. Let $\{\upsilon_{n}\}$ be a sequence such that $\upsilon_{n} \rightarrow \upsilon$ in $\widetilde{W}.$ For $(\varkappa,y) \in \widetilde{\boldsymbol{\varOmega}}$ we have
		\begin{equation*}
			\begin{split}
				&{\mathbb{D}}(\mathbb{A}\upsilon_{n}(\varkappa,y),\mathbb{A}\upsilon(\varkappa,y)) \\&\leq \frac{\rho_{1}^{1-\varphi_{1}}\rho_{2}^{1-\varphi
						_2}}{\Gamma(\varphi_{1})\Gamma(\varphi_{2})} \int_{0}^{\varkappa}\int_{0}^{y}\frac{s^{\rho_{1}-1}t^{\rho_{2}-1}}{(\varkappa^{\rho_{1}}-s^{\rho_{1}})^{1-\varphi_{1}}(y^{\rho_{2}}-t^{\rho_{2}})^{1-\varphi_{2}}}\times\\&  \sup_{(s,t) \in {\boldsymbol{\varOmega}}}{\mathbb{D}}\big(\mathcal{F}(s,t,\upsilon_{n}(s,t),\mathcal{F}(s,t,\upsilon(s,t))\big)dtds\\\\&\leq\frac{\rho_{1}^{1-\varphi_{1}}\rho_{2}^{1-\varphi
						_2}}{\Gamma(\varphi_{1})\Gamma(\varphi_{2})}\sup_{(s,t) \in {\boldsymbol{\varOmega}}}{\mathbb{D}}\big(\mathcal{F}(s,t,\upsilon_{n}(s,t),\mathcal{F}(s,t,\upsilon(s,t))\big)\times\\& \int_{0}^{\varkappa}\int_{0}^{y}\frac{s^{\rho_{1}-1}t^{\rho_{2}-1}}{(\varkappa^{\rho_{1}}-s^{\rho_{1}})^{1-\varphi_{1}}(y^{\rho_{2}}-t^{\rho_{2}})^{1-\varphi_{2}}}dtds
					\\\\&\leq\frac{\varkappa^{\rho_{1}\varphi_{1}}y^{\rho_{2}\varphi_{2}}}{\rho_{1}^{\varphi_{1}}\rho_{2}^{\alpha
						_2}\Gamma(\varphi_{1}+1)\Gamma(\varphi_{2}+1)}\sup_{(s,t) \in {\boldsymbol{\varOmega}}}{\mathbb{D}}\big(\mathcal{F}(s,t,\upsilon_{n}(s,t),\mathcal{F}(s,t,\upsilon(s,t))\big)\\&\leq\frac{a^{\rho_{1}\varphi_{1}}b^{\rho_{2}\varphi_{2}}}{\rho_{1}^{\alpha_{1}}\rho_{2}^{\varphi
						_2}\Gamma(\varphi_{1}+1)\Gamma(\varphi_{2}+1)}\sup_{(s,t) \in {\boldsymbol{\varOmega}}}{\mathbb{D}}\big(\mathcal{F}(s,t,\upsilon_{n}(s,t),\mathcal{F}(s,t,\upsilon(s,t))\big),
			\end{split}
		\end{equation*}
	since $\mathcal{F}$ is continuous, we have 
	\[	{\mathbb{D}}(\mathbb{A}\upsilon_{n}(\varkappa,y),\mathbb{A}\upsilon(\varkappa,y)) \rightarrow 0~ \text{as} ~n\rightarrow \infty.\] 
		Step 2. We show that $\mathbb{A}$ is relatively compact.Let $(\varkappa_{1},y_{1}),(\varkappa_{2},y_{2}) \in \widetilde{\boldsymbol{\varOmega}}$ such that $\varkappa_{1}<\varkappa_{2}$ and $y_{1}<y_{2}$. Then we have
		\begin{equation*}
				\begin{split}
				&{\mathbb{D}}(\mathbb{A}\upsilon(\varkappa_{2},y_{2}),\mathbb{A}\upsilon(\varkappa_{1},y_{1})) \\&\leq {\mathbb{D}}(h(\varkappa_{2},y_{2}),h(\varkappa_{1},y_{1}))+ \frac{\rho_{1}^{1-\varphi_{1}}\rho_{2}^{1-\varphi
						_2}}{\Gamma(\varphi_{1})\Gamma(\alpha_{2})} \times\\&{\mathbb{D}}\bigg(\int_{0}^{\varkappa_{2}}\int_{0}^{y_{2}}\frac{s^{\rho_{1}-1}t^{\rho_{2}-1}\mathcal{F}(s,t,\upsilon)dtds}{(\varkappa_{2}^{\rho_{1}}-s^{\rho_{1}})^{1-\varphi_{1}}(y_{2}^{\rho_{2}}-t^{\rho_{2}})^{1-\varphi_{2}}}  ,\int_{0}^{\varkappa_{1}}\int_{0}^{y_{1}}\frac{s^{\rho_{1}-1}t^{\rho_{2}-1}\mathcal{F}(s,t,\upsilon)dtds}{(\varkappa_{1}^{\rho_{1}}-s^{\rho_{1}})^{1-\varphi_{1}}(y_{1}^{\rho_{2}}-t^{\rho_{2}})^{1-\varphi_{2}}} \bigg)
						\end{split}
				\end{equation*}
			\begin{equation*}
				\begin{split}
				\leq&{\mathbb{D}}(h(\varkappa_{2},y_{2}),h(\varkappa_{1},y_{1}))+ \frac{\rho_{1}^{1-\varphi_{1}}\rho_{2}^{1-\varphi
						_2}}{\Gamma(\varphi_{1})\Gamma(\varphi_{2})}\times\\& {\mathbb{D}}\bigg(\int_{0}^{\varkappa_{1}}\int_{0}^{y_{1}}\frac{s^{\rho_{1}-1}t^{\rho_{2}-1}\mathcal{F}(s,t,\upsilon)dtds}{(\varkappa_{2}^{\rho_{1}}-s^{\rho_{1}})^{1-\alpha_{1}}(y_{2}^{\rho_{2}}-t^{\rho_{2}})^{1-\alpha_{2}}}  ,\int_{0}^{\varkappa_{1}}\int_{0}^{y_{1}}\frac{s^{\rho_{1}-1}t^{\rho_{2}-1}\mathcal{F}(s,t,\upsilon)dtds}{(\varkappa_{1}^{\rho_{1}}-s^{\rho_{1}})^{1-\varphi_{1}}(y_{1}^{\rho_{2}}-t^{\rho_{2}})^{1-\varphi_{2}}} \bigg)\\&+{\mathbb{D}}\bigg(\int_{\varkappa_{1}}^{\varkappa_{2}}\int_{0}^{y_{1}}\frac{s^{\rho_{1}-1}t^{\rho_{2}-1}\mathcal{F}(s,t,\upsilon)dtds}{(\varkappa_{2}^{\rho_{1}}-s^{\rho_{1}})^{1-\varphi_{1}}(y_{2}^{\rho_{2}}-t^{\rho_{2}})^{1-\varphi_{2}}},0\bigg)\\&+{\mathbb{D}}\bigg(\int_{0}^{\varkappa_{1}}\int_{y_{1}}^{y_{2}}\frac{s^{\rho_{1}-1}t^{\rho_{2}-1}\mathcal{F}(s,t,\upsilon)dtds}{(\varkappa_{2}^{\rho_{1}}-s^{\rho_{1}})^{1-\alpha_{1}}(y_{2}^{\rho_{2}}-t^{\rho_{2}})^{1-\alpha_{2}}},0\bigg)\\&+{\mathbb{D}}\bigg(\int_{\varkappa_{1}}^{\varkappa_{2}}\int_{y_{1}}^{y_{2}}\frac{s^{\rho_{1}-1}t^{\rho_{2}-1}\mathcal{F}(s,t,\upsilon)dtds}{(\varkappa_{2}^{\rho_{1}}-s^{\rho_{1}})^{1-\varphi_{1}}(y_{2}^{\rho_{2}}-t^{\rho_{2}})^{1-\varphi_{2}}},0\bigg)
					 	\end{split}
				 \end{equation*}
			 \begin{equation*}
			 	\begin{split}
			 	\leq&
					 {\mathbb{D}}(h(\varkappa_{2},y_{2}),h(\varkappa_{1},y_{1}))+ \frac{M\rho_{1}^{1-\varphi_{1}}\rho_{2}^{1-\varphi
						_2}}{\Gamma(\varphi_{1})\Gamma(\varphi_{2})}\times\\&\bigg[\int_{0}^{\varkappa_{1}}\int_{0}^{y_{1}}\big[\frac{s^{\rho_{1}-1}t^{\rho_{2}-1}}{(\varkappa_{2}^{\rho_{1}}-s^{\rho_{1}})^{1-\varphi_{1}}(y_{2}^{\rho_{2}}-t^{\rho_{2}})^{1-\varphi_{2}}} -\frac{s^{\rho_{1}-1}t^{\rho_{2}-1}}{(\varkappa_{1}^{\rho_{1}}-s^{\rho_{1}})^{1-\varphi_{1}}(y_{1}^{\rho_{2}}-t^{\rho_{2}})^{1-\varphi_{2}}}\big]dtds\\&+\int_{\varkappa_{1}}^{\varkappa_{2}}\int_{0}^{y_{1}}\frac{s^{\rho_{1}-1}t^{\rho_{2}-1}dtds}{(\varkappa_{2}^{\rho_{1}}-s^{\rho_{1}})^{1-\alpha_{1}}(y_{2}^{\rho_{2}}-t^{\rho_{2}})^{1-\alpha_{2}}}+\int_{0}^{\varkappa_{1}}\int_{y_{1}}^{y_{2}}\frac{s^{\rho_{1}-1}t^{\rho_{2}-1}dtds}{(\varkappa_{2}^{\rho_{1}}-s^{\rho_{1}})^{1-\varphi_{1}}(y_{2}^{\rho_{2}}-t^{\rho_{2}})^{1-\varphi_{2}}}\\&+\int_{\varkappa_{1}}^{\varkappa_{2}}\int_{y{1}}^{y_{2}}\frac{s^{\rho_{1}-1}t^{\rho_{2}-1}dtds}{(\varkappa_{2}^{\rho_{1}}-s^{\rho_{1}})^{1-\varphi_{1}}(y_{2}^{\rho_{2}}-t^{\rho_{2}})^{1-\varphi_{2}}}\bigg]\\\leq&{\mathbb{D}}(h(\varkappa_{2},y_{2}),h(\varkappa_{1},y_{1}))+ \frac{M}{\rho_{1}^{\varphi_{1}}\rho_{2}^{\varphi
						_2}\Gamma(\varphi_{1}+1)\Gamma(\varphi_{2}+1)}\big(\varkappa_{2}^{\rho_{1}\varphi_{1}}y_{2}^{\rho_{2}\varphi_{2}}-\varkappa_{1}^{\rho_{1}\varphi_{1}}y_{1}^{\rho_{2}\varphi_{2}}\big).
			\end{split}
		\end{equation*}
	As $\varkappa_{1}<\varkappa_{2}$ and $y_{1}<y_{2}$ then ${\mathbb{D}}(\mathbb{A}\upsilon(\varkappa_{2},y_{2}),\mathbb{A}\upsilon(\varkappa_{1},y_{1}))$ tends to zero. This implies that $\mathbb{A}$ is equicontinuous on $C[\widetilde{\boldsymbol{\varOmega}},\hat{E_{f}}]$.
	
	Now, to prove that $\mathbb{A}$ is relatively compact  we must prove
	$\mathbb{A}$ is level-equicontinuous and a compact-supported.
 Since $\mathcal{F}$ is a compact mapping, $\mathcal{F}:\widetilde{\boldsymbol{\varOmega}}\times\widetilde{W}\rightarrow \hat{E_{f}}$ is relatively compact  and from theorem (\ref{s1*}) 
	$\mathcal{F}$ is level-equicontinuous. Then for each $\varepsilon>0$, $q\in [0,1]$ there exists $\delta>0$ such that from $|r-q|<\delta$ we get
	$$d_{H}(\big[\mathcal{F}(\varkappa,y,\upsilon(\varkappa,y))\big]^{r},\big[\mathcal{F}(\varkappa,y,\upsilon(\varkappa,y))\big]^{q})<\frac{\rho_{1}^{\varphi_{1}}\rho_{2}^{\varphi
			_{2}}\Gamma(\varphi_{1}+1)\Gamma(\varphi_{2}+1)}{2a^{\rho_{1}\varphi_{1}}b^{\rho_{2}\varphi_{2}}}\varepsilon$$		
	and $d_{H}(\big[h(\varkappa,y)\big]^{r},\big[h(\varkappa,y)\big]^{q}))<\frac{\varepsilon}{2}$. Hence, when $|r-q|<\delta$ we have
	\begin{equation*}
		\begin{aligned}
			d_{H}&(\big[\mathbb{A}\upsilon(\varkappa,y)\big]^{r},\big[\mathbb{A}\upsilon(\varkappa,y)\big]^{q})\leq d_{H}(\big[h(\varkappa,y)\big]^{r},\big[h(\varkappa,y)\big]^{q}))\\&+\frac{\rho_{1}^{1-\varphi_{1}}\rho_{2}^{1-\varphi	_{2}}}{\Gamma(\varphi_{1})\Gamma(\varphi_{2})} \int_{0}^{\varkappa}\int_{0}^{y}\frac{s^{\rho_{1}-1}t^{\rho_{2}-1}}{(\varkappa^{\rho_{1}}-s^{\rho_{1}})^{1-\varphi_{1}}(y^{\rho_{2}}-t^{\rho_{2}})^{1-\varphi_{2}}}{d_{H}}(\big[\mathcal{F}(s,t,\upsilon)\big]^{r},\big[\mathcal{F}(s,t,\upsilon)\big]^{q})dtds\\\\\leq&\frac{\epsilon}{2}+\frac{\rho_{1}\rho_{2}\varepsilon\Gamma(\varphi_{1}+1)\Gamma(\varphi_{2}+1)}{2a^{\alpha_{1}}b^{\alpha_{2}}\Gamma(\alpha_{1})\Gamma(\varphi_{2})}\int_{0}^{\varkappa}\int_{0}^{y}\frac{s^{\rho_{1}-1}t^{\rho_{2}-1}}{(\varkappa^{\rho_{1}}-s^{\rho_{1}})^{1-\alpha_{1}}(y^{\rho_{2}}-t^{\rho_{2}})^{1-\alpha_{2}}}dtds\\\\\leq&\frac{\epsilon}{2}+\frac{\rho_{1}\rho_{2}\varepsilon\Gamma(\varphi_{1}+1)\Gamma(\varphi_{2}+1)}{2a^{\rho_{1}\varphi_{1}}b^{\rho_{2}\varphi_{2}}\Gamma(\varphi_{1})\Gamma(\varphi_{2})}\frac{\varkappa^{\rho_{1}\varphi_{1}}y^{\rho_{2}\varphi_{2}}}{\rho_{1}\rho_{2}\varphi_{1}\varphi_{2}}\leq \epsilon.
		\end{aligned}
	\end{equation*}
	Thus, $\mathbb{A}$ is level-equicontinuous.
	
	Since $\mathcal{F}$ and $h$ are relative compact, then from theorem (\ref{s1*}) $\mathcal{F}$ and $h$ are support-compact and level-equicontinuous. Hence, there exists a compact subsets $M_{1},M_{2} \subset \mathbb{R}$ such that
	$\big[\mathcal{F} (\varkappa,y,\upsilon(\varkappa,y))\big]^{0} \subseteq M_{1}$  and 
	$\big[h(\varkappa, y)\big]^{0} \subseteq M_{2}$ for every $(\varkappa,y,\upsilon) \in \widetilde{\boldsymbol{\varOmega}}\times\widetilde{W}.$ Thus, we obtain
	\begin{equation*}
		\begin{aligned}
			&\big[\mathbb{A}(\upsilon(\varkappa,y))\big]^{0}\\&=\bigg[h(\varkappa,y)+\frac{\rho_{1}^{1-\varphi_{1}}\rho_{2}^{1-\varphi_{2}}}{\Gamma(\varphi_{1})\Gamma(\varphi_{2})} \int_{0}^{\varkappa}\int_{0}^{y}\frac{s^{\rho_{1}-1}t^{\rho_{2}-1}}{(\varkappa^{\rho_{1}}-s^{\rho_{1}})^{1-\varphi_{1}}(y^{\rho_{2}}-t^{\rho_{2}})^{1-\varphi_{2}}}\mathcal{F}(s,t,\upsilon)dtds\bigg]^{0}\\&=\big[h(\varkappa, y)\big]^{0}+\frac{\rho_{1}^{1-\varphi_{1}}\rho_{2}^{1-\varphi_{2}}}{\Gamma(\varphi_{1})\Gamma(\varphi_{2})} \int_{0}^{\varkappa}\int_{0}^{y}\frac{s^{\rho_{1}-1}t^{\rho_{2}-1}}{(\varkappa^{\rho_{1}}-s^{\rho_{1}})^{1-\varphi_{1}}(y^{\rho_{2}}-t^{\rho_{2}})^{1-\varphi_{2}}}\big[\mathcal{F} (\varkappa,y,\upsilon)\big]^{0}dtds\\&\subseteq M_{2}+\frac{M_{1}\varkappa^{\rho_{1}\varphi_{1}}y^{\rho_{2}\varphi_{2}}}{\rho_{1}^{\varphi_{1}}\rho_{2}^{\varphi_{2}}\Gamma(\varphi_{1}+1)\Gamma(\varphi_{2}+1)}.
		\end{aligned}
	\end{equation*}
	Since $\varkappa^{\rho_{1}\varphi_{1}}y^{\rho_{2}\varphi_{2}}$ is bounded on $\widetilde{\boldsymbol{\varOmega}}$, then there exists a compact set $M_{0} \subseteq\mathbb{R}$ such that\\ 
	$\big[\mathbb{A}(\upsilon(\varkappa,y))\big]^{0}\subseteq M_{0}$.
	Hence, $\mathbb{A}$ is compact-supported. Therefor, by Ascoli-Arzelá theorem $\mathbb{A}$ is relatively compact. \\According to the above steps with Schouder's theorem, we deduce that $\mathbb{A}$ has at least one fixed point $\upsilon$ which is a solution to the $\mathcal{CK}$ system (\ref{i1}).
\end{proof} 
In the following theorem, we discuss the uniqueness results for the problem (\ref{i1}).
	\begin{theorem} \label{s4}
		Let $\upsilon, \hat{\upsilon} \in C[{\boldsymbol{\varOmega}},\hat{E_{f}}]$, for all $(\varkappa,y) \in{\boldsymbol{\varOmega}}$ there exists a  constant $R>0$ such that
	\[{\mathbb{D}}(\mathcal{F}(\varkappa,y,\upsilon),\mathcal{F}(\varkappa,y,\hat{\upsilon}))\leq R ~{\mathbb{D}}(\upsilon,\hat{\upsilon}).\] 
	Then the problem(\ref{i1}) has a unique solution if 
	\[\varXi=\frac{R}{\Gamma(\varphi_{1}+1)\Gamma(\varphi_{2}+1)}
	\bigg(\frac{a^{\rho_{1}}}{\rho_{1}}\bigg)^{\varphi_{1}}\bigg(\frac{b^{\rho_{2}}}{\rho_{2}}\bigg)^{\varphi_{2}}<1.\]
	\end{theorem}
	\begin{proof}
		Let $\mathbb{A}$ be an operator defined in Theorem (\ref{s3}). We will show that $\mathbb{A}$ has a unique  fuzzy solution.
 	Let $\upsilon, \hat{\upsilon} \in C[\widetilde{\boldsymbol{\varOmega}},\hat{E_{f}}]$, for all $(\varkappa,y) \in\widetilde{\boldsymbol{\varOmega}}$. Then
		\begin{equation*}
			\begin{aligned}
				{\mathbb{D}}(\mathbb{A}\upsilon(\varkappa,y),\mathbb{A}\hat{\upsilon}(\varkappa,y)) &\leq\frac{\rho_{1}^{1-\varphi_{1}}\rho_{2}^{1-\varphi	_{2}}}{\Gamma(\varphi_{1})\Gamma(\varphi_{2})} \int_{0}^{\varkappa}\int_{0}^{y}\frac{s^{\rho_{1}-1}t^{\rho_{2}-1}{\mathbb{D}}(\mathcal{F}(s,t,\upsilon),\mathcal{F}(s,t,\hat{\upsilon}))}{(\varkappa^{\rho_{1}}-s^{\rho_{1}})^{1-\varphi_{1}}(y^{\rho_{2}}-t^{\rho_{2}})^{1-\varphi_{2}}}dtds\\ &\leq \frac{R\rho_{1}^{1-\varphi_{1}}\rho_{2}^{1-\varphi	_{2}}H^{*}(\varkappa,y)}{\Gamma(\varphi_{1})\Gamma(\varphi_{2})} \int_{0}^{\varkappa}\int_{0}^{y}\frac{s^{\rho_{1}-1}t^{\rho_{2}-1}}{(\varkappa^{\rho_{1}}-s^{\rho_{1}})^{1-\varphi_{1}}(y^{\rho_{2}}-t^{\rho_{2}})^{1-\varphi_{2}}}dtds\\&\leq  \frac{R H^{*}(\varkappa,y)}{\Gamma(\alpha_{1}+1)\Gamma(\varphi_{2}+1)}\bigg(\frac{a^{\rho_{1}}}{\rho_{1}}\bigg)^{\varphi_{1}}\bigg(\frac{b^{\rho_{2}}}{\rho_{2}}\bigg)^{\varphi_{2}}.
			\end{aligned}
		\end{equation*} 
	Consequently,
		\[H^{*}(\mathbb{A}\upsilon(\varkappa,y),\mathbb{A}\hat{\upsilon}(\varkappa,y)) \leq\varXi H^{*}(\varkappa,y).\]
	Hence $\mathbb{A}$ is a contraction mapping. Thus, by Banach contraction principle $\mathbb{A}$ has a unique fixed point $\upsilon\in C[{\boldsymbol{\varOmega}},\hat{E_{f}}]$.
	\end{proof} 
	\begin{example}\label{3l2}
		Consider the following problem:
			\begin{equation}
			\label{s5}
			\left\{
			\begin{array}{cc}
				^{C}\mathfrak{D}_{0^{+}}^{\varphi,\rho}\upsilon(\varkappa,y) = \frac{\varkappa y \upsilon(\varkappa,y)}{2(1+\upsilon(\varkappa,y))},  \\\\
				\upsilon(\varkappa,0)=\mathcal{K} \varkappa, 
				\upsilon(0,y)=\mathcal{K} y^{2}, \upsilon(0,0)=0,
			\end{array}\right .
		\end{equation}
			where $(\varkappa,y)\in (0,\frac{1}{2}]\times(0,1]$,  $\varphi=\frac{1}{2}$,$\rho=\frac{3}{2}$ and $\mathcal{K}=(1,2,3)$ is a triangle fuzzy number.
		\end{example}
		We have\\
		$$\mathcal{F}(\varkappa,y,\upsilon(\varkappa,y))=\frac{\varkappa y \upsilon(\varkappa,y)}{2(1+\upsilon(\varkappa,y))}.$$ 
		Now for $(\varkappa,y)\in (0,\frac{1}{2}]\times(0,1]$  we can see that the function $\mathcal{F}$ satisfies the condition of Theorem (\ref{s4}) as follows:
		\begin{equation*}
			\begin{split}
				{\mathbb{D}}(\mathcal{F}(\varkappa,y,\upsilon),\mathcal{F}(\varkappa,y,\hat{\upsilon}))&={\mathbb{D}}\big(\frac{\varkappa y \upsilon(\varkappa,y)}{2(1+\upsilon(\varkappa,y))},\frac{\varkappa y\hat{\upsilon}(\varkappa,y)}{2(1+\hat{\upsilon}(\varkappa,y))}\big)\\&=\frac{\varkappa t}{2}  {\mathbb{D}}\big (\frac{\upsilon}{1+\upsilon},\frac{\hat{\upsilon}}{1+\hat{\upsilon}}\big)\\&\leq \frac{1}{4}{\mathbb{D}}(\upsilon,\hat{\upsilon}).
			\end{split}
		\end{equation*}	
Also, for $\varphi_{1}=\varphi_{2}=0.5$ and $\rho_{1}=\rho_{2}=1.5$ we get
	\begin{equation*}
	\begin{split}
\varXi&=\frac{R}{\Gamma(\alpha_{1}+1)\Gamma(\alpha_{2}+1)}
\bigg(\frac{a^{\rho_{1}}}{\rho_{1}}\bigg)^{\varphi_{1}}\bigg(\frac{b^{\rho_{2}}}{\rho_{2}}\bigg)^{\varphi_{2}}\\&=\frac{1}{4\Gamma(1.5)\Gamma(1.5)}\bigg(\frac{0.5^{1.5}}{1.5}\bigg)^{0.5}\bigg(\frac{1^{1.5}}{1.5}\bigg)^{0.5}\\&\backsimeq 0.35689 <1.
	\end{split}
\end{equation*}	
Since the assumption of Theorem (\ref{s4}) are obtained, therefore problem (\ref{s5}) has a unique solution. 
	\section{ Fuzzy Fractional Coupled Partial
		Differential Equations with  Caputo- Katogumpola g$\mathcal{H}$ derivative}
	In this section, we will prove the existence and uniqueness of a solutions of the coupled system for fuzzy fractional partial differential equations (\ref{i2}). For all $(\varkappa,y)\in \boldsymbol{\varOmega}$.
	 let
	 \begin{equation}
			h(\varkappa,y)=\xi_{1}(\varkappa)+\xi_{2}(y)\ominus\xi_{1}(0),
				\end{equation}
		\begin{equation}
		\mathfrak{g}(\varkappa,y)=\eta_{1}(\varkappa)+\eta_{2}(y)\ominus\eta_{1}(0),
		\end{equation}
		where $\xi_{1}$, $\xi_{2}$, $\eta_{1}$ and $\eta_{2}$ are the
	given  fuzzy valued functions such that  $\xi_{2}(y)\ominus\xi_{1}(0)$ and $\eta_{2}(y)\ominus\eta_{1}(0)$ exist.

Now we introduce an auxiliary lemma that will have significant importance in this part of the paper.
	\begin{lemma}\label{c1}
		Let $\mathcal{F},\mathcal{G} : \boldsymbol{\varOmega} \times \hat{E_{f}} \rightarrow \hat{E_{f}}$. For any $\upsilon,\omega\in C(\boldsymbol{\varOmega},\hat{E_{f}})$,   $\varphi=(\varphi_{1},\varphi_{2})\in(0,1)$, $\psi=(\psi_{1},\psi_{2})\in(0,1)$ and $\rho=(\rho_{1},\rho_{2})>0$, then one of the following  coupled integral equations are equivalent to the coupled system (\ref{i2}) with (\ref{i3}):
		\begin{equation}
			\label{c2}
			\left\{
			\begin{array}{cc}
				\upsilon(\varkappa,y)=h(\varkappa,y)+ \mathfrak{I}^{\varphi,\rho}_{0^{+}}\big[\mathcal{F}(\varkappa,y,\upsilon(\varkappa,y),\omega(\varkappa,y))\big],\\\\
				\omega(\varkappa,y)=\mathfrak{g}(\varkappa,y)+ \mathfrak{I}^{\psi,\rho}_{0^{+}}\big[\mathcal{G}(\varkappa,y,\upsilon(\varkappa,y),\omega(\varkappa,y))\big],
			\end{array}\right .
		\end{equation}
		or
		\begin{equation}
			\label{c3}
			\left\{
			\begin{array}{cc}
				\upsilon(\varkappa,y)=h(\varkappa,y)\ominus(-1) \mathfrak{I}^{\varphi,\rho}_{0^{+}}\big[\mathcal{F}(\varkappa,y,\upsilon(\varkappa,y),\omega(\varkappa,y))\big],\\\\
				\omega(\varkappa,y)=\mathfrak{g}(\varkappa,y)\ominus(-1) \mathfrak{I}^{\psi,\rho}_{0^{+}}\big[\mathcal{G}(\varkappa,y,\upsilon(\varkappa,y),\omega(\varkappa,y))\big].
			\end{array}\right .
		\end{equation}
	\end{lemma}	
	\begin{proof} 
	The equivalent Volterra integral equation for coupled system can be obtained  by applying the  same process of single system in Lemma (\ref{s0}).
\end{proof}
Next, by using Schauder fixed point in Banach spaces, we present existence and uniqueness theorems only for solutions (\ref{c2}), where we can proved solutions (\ref{c2}) in similar way.
		\begin{theorem} \label{c3}
		Let $k$ is a positive constant. For all $\varphi = (\varphi_{1}, \varphi_{2}), \psi=(\psi_{1},\psi_{2}) \in (0, 1]^{2},$ there exist constants $\varkappa_{1},\varkappa_{2},y_{1},y_{2}>0$ such that $h,\mathfrak{g}\in C[{\boldsymbol{\varOmega}}, \mathcal{B}(0,\frac{k}{2})]$ and\\
		$$\varkappa_{1}^{\rho_{1}}y_{1}^{\rho_{2}}\leq\big[\frac{k\Gamma(1+\varphi_{1})\Gamma(1+\varphi_{2})}{2\mathcal{M}_{1}}\rho_{1}^{\varphi_{1}}\rho_{2}^{\varphi_{2}}\big]^\frac{1}{\varphi_{1}\varphi_{2}},~ \varkappa_{2}^{\rho_{1}}y_{2}^{\rho_{2}}\leq\big[\frac{k\Gamma(1+\psi_{1})\Gamma(1+\psi_{2})}{2\mathcal{M}_{2}}\rho_{1}^{\psi_{1}}\rho_{2}^{\psi_{2}}\big]^\frac{1}{\psi_{1}\psi_{2}}.$$  Also, let $S_{1}=min\{\varkappa_{1},\varkappa_{2},a\}$, $T_{1}=min\{y_{1},y_{1},b\}$.
		Define
		\[\mathcal{M}_{1}=\sup_{(\varkappa,y,\upsilon,\omega) \in {\boldsymbol{\varOmega}}, \times \mathcal{B}(0,k)\times\mathcal{B}(0,k)}H^{*}[\mathcal{F}(\varkappa,y,\upsilon(\varkappa,y),\omega(\varkappa,y)),0] ,\]
			\[\mathcal{M}_{2}=\sup_{(\varkappa,y,\upsilon,\omega) \in {\boldsymbol{\varOmega}}, \times \mathcal{B}(0,k)\times\mathcal{B}(0,k)}H^{*}[\mathcal{G}(\varkappa,y,\upsilon(\varkappa,y),\omega(\varkappa,y)),0],\]
		where $\mathcal{M}_{1},\mathcal{M}_{2}$ is a positive constants. Then for all $(\varkappa,y)\in\Lambda=[0,S_{1}]\times[0,T_{1}]$  there exists a function $(\upsilon,\omega) \in C[\Lambda,\hat{E_{f}}] \times C[\Lambda,\hat{E_{f}}]$ that solve the problem (\ref{i2}).
	\end{theorem}
\begin{proof}
	Consider the operator $\Im:\widetilde{W_{1}}\rightarrow \widetilde{W_{1}}$ defined as
	\begin{equation}\label{c0}
		\Im	\upsilon(\varkappa,y)=\binom{\Im_{1}	\upsilon(\varkappa,y)}{\Im_{2}\omega(\varkappa,y)},
	\end{equation}
	where
	\begin{equation*}
		\Im_{1}	\upsilon(\varkappa,y)=h(\varkappa,y)+ \mathfrak{I}^{\varphi,\rho}_{0^{+}}\big[\mathcal{F}(\varkappa,y,\upsilon(\varkappa,y),\omega(\varkappa,y))\big],
			\end{equation*}
		\begin{equation*}
		\Im_{2}	\omega(\varkappa,y)=\mathfrak{g}(\varkappa,y)+ \mathfrak{I}^{\psi,\rho}_{0^{+}}\big[\mathcal{G}(\varkappa,y,\upsilon(\varkappa,y),\omega(\varkappa,y))\big].
	\end{equation*}
For $(\varkappa,y) \in \Lambda$ we define\\ $\widetilde{W_{1}}=\{(\upsilon,\omega)\in C[\Lambda,\hat{E_{f}}] \times C[\Lambda,\hat{E_{f}}] : max \{H^{*}(\upsilon,0),H^{*}(\omega,0)\}] \leq k \}$.\\
Also, for each	$\binom{	\upsilon_{1}(\varkappa,y)}{\omega_{1}(\varkappa,y)},\binom{	\upsilon_{1}(\varkappa,y)}{\omega_{1}(\varkappa,y)} \in\widetilde{W_{1}}$,\\
\begin{equation}\label{c4}
	\begin{split}
		&	H^{*}(\Im(\upsilon_{1},\omega_{1}),\Im(\upsilon_{2},\omega_{2}))\\&=max\{	H^{*}(\Im_{1}(\upsilon_{1},\omega_{1}),\Im_{1}(\upsilon_{2},\omega_{2})),H^{*}(\Im_{2}(\upsilon_{1},\omega_{1}),\Im_{2}(\upsilon_{2},\omega_{2}))\}\\&=max \big\{\sup_{(\varkappa,y) \in \Lambda}\mathbb{D}(\Im_{1}(\upsilon_{1},\omega_{1}),\Im_{1}(\upsilon_{2},\omega_{2})),\sup_{(\varkappa,y) \in \Lambda}\mathbb{D}(\Im_{2}(\upsilon_{1},\omega_{1}),\Im_{2}(\upsilon_{2},\omega_{2})) \big\}.
	\end{split}
\end{equation}
		Now we show that the operator $\Im$ has a fixed point by the following steps\\ 
	Step 1. $\Im$ is continuous. \\
	Firstly, we show that $\Im$ maps the set $\widetilde{W_{1}}$ in to itself, for every $(\upsilon,\omega) \in \widetilde{W_{1}}$ and $(\varkappa,y) \in \Lambda$ we have
	\begin{equation*}
		\begin{split}
			&{\mathbb{D}}(\Im_{1}(\upsilon(\varkappa,y),\omega(\varkappa,y)),0)\\ \leq &{\mathbb{D}}(h(\varkappa,y),0)+ \frac{\rho_{1}^{1-\varphi_{1}}\rho_{2}^{1-\varphi
					_2}}{\Gamma(\varphi_{1})\Gamma(\varphi_{2})} \int_{0}^{\varkappa}\int_{0}^{y}\frac{s^{\rho_{1}-1}t^{\rho_{2}-1}{\mathbb{D}}\big(\mathcal{F}(s,t,\upsilon(s,t),\omega(\varkappa,y),0)\big)dtds}{(\varkappa^{\rho_{1}}-s^{\rho_{1}})^{1-\varphi_{1}}(y^{\rho_{2}}-t^{\rho_{2}})^{1-\varphi_{2}}} \\\\\leq&{\mathbb{D}}(h(\varkappa,y),0)+\frac{\rho_{1}^{1-\varphi_{1}}\rho_{2}^{1-\varphi
					_2}\mathcal{M}_{1}}{\Gamma(\varphi_{1})\Gamma(\varphi_{2})} \int_{0}^{\varkappa}\int_{0}^{y}\frac{s^{\rho_{1}-1}t^{\rho_{2}-1}}{(\varkappa^{\rho_{1}}-s^{\rho_{1}})^{1-\varphi_{1}}(y^{\rho_{2}}-t^{\rho_{2}})^{1-\varphi_{2}}}dtds\\\leq&{\mathbb{D}}(h(\varkappa,y),0)+\frac{\mathcal{M}_{1}\varkappa^{\rho_{1}\varphi_{1}}y^{\rho_{1}\varphi_{2}}}{\rho_{1}^{\varphi_{1}}\rho_{2}^{\varphi
					_2}\Gamma(\varphi_{1}+1)\Gamma(\varphi_{2}+1)},
		\end{split}
	\end{equation*}
which implies
\begin{equation}\label{c5}
	{\mathbb{D}}(\Im_{1}(\upsilon(\varkappa,y),\omega(\varkappa,y)),0)\leq\frac{k}{2}+\frac{k}{2}=k.
\end{equation}
Similarly, we obtain
\begin{equation}\label{c6}
	\mathbb{D}(\Im_{2}(\upsilon(\varkappa,y),\omega(\varkappa,y)),0) \leq k.
\end{equation}
For all $(\upsilon,\omega)\in \widetilde{W_{1}}$ and from (\ref{c4}) we get
\begin{equation}\label{c7}
	\mathbb{D}(\Im(\upsilon_{1},\omega_{1}),0))\\=max\{	\mathbb{D}(\Im_{1}(\upsilon_{1},\omega_{1}),0),\mathbb{D}(\Im_{2}(\upsilon_{1},\omega_{1}),0)\}.
\end{equation}
Combining (\ref{c5}), (\ref{c6}) with (\ref{c7}), we get
$H^{*}(\Im(\upsilon_{1},\omega_{1}),0)\leq k.$
	Hence $\Im(\upsilon,\omega) \in \widetilde{W_{1}}$.
	
	Now we proof that $\Im$ is continuous. Let $\{(\upsilon_{n},\omega_{n})\}$ be a sequence such that\\ $(\upsilon_{n},\omega_{n}) \rightarrow (\upsilon,\omega)$ in $\widetilde{W_{1}}.$\\ For $(\varkappa,y) \in \Lambda$ we have
	\begin{equation}\label{c8}
	\begin{split}
		&{\mathbb{D}}(\Im_{1}(\upsilon_{n},\omega_{n}),\Im_{1}(\upsilon,\omega) \\&\leq \frac{\rho_{1}^{1-\varphi_{1}}\rho_{2}^{1-\varphi
				_2}}{\Gamma(\varphi_{1})\Gamma(\varphi_{2})} \int_{0}^{\varkappa}\int_{0}^{y}\frac{s^{\rho_{1}-1}t^{\rho_{2}-1}}{(\varkappa^{\rho_{1}}-s^{\rho_{1}})^{1-\varphi_{1}}(y^{\rho_{2}}-t^{\rho_{2}})^{1-\varphi_{2}}}\times\\&  \sup_{(s,t) \in \Lambda}{\mathbb{D}}\big(\mathcal{F}(s,t,\upsilon_{n}(s,t),\omega_{n}(s,t)),\mathcal{F}(s,t,\upsilon(s,t),\omega(s,t))\big)dtds\\&\leq\frac{a^{\rho_{1}\varphi_{1}}b^{\rho_{2}\varphi_{2}}}{\rho_{1}^{\varphi_{1}}\rho_{2}^{\varphi
				_2}\Gamma(\varphi_{1}+1)\Gamma(\varphi_{2}+1)}\times\\&\sup_{(s,t) \in \Lambda}{\mathbb{D}}\big(\mathcal{F}(s,t,\upsilon_{n}(s,t),\omega_{n}(s,t)),\mathcal{F}(s,t,\upsilon(s,t),\omega(s,t))\big),
	\end{split}
\end{equation}
since $\mathcal{F}$ is continuous. Hence $\Im$ is continuous. Similarly
\begin{equation}\label{c9}
	\begin{split}
	{\mathbb{D}}(\Im_{2}(\upsilon_{n},\omega_{n}),\Im_{2}(\upsilon,\omega)&\leq\frac{a^{\rho_{1}\varphi_{1}}b^{\rho_{2}\varphi_{2}}}{\rho_{1}^{\varphi_{1}}\rho_{2}^{\varphi
			_2}\Gamma(\varphi_{1}+1)\Gamma(\varphi_{2}+1)}\times\\&\sup_{(s,t) \in \Lambda}{\mathbb{D}}\big(\mathcal{G}(s,t,\upsilon_{n}(s,t),\omega_{n}(s,t)),\mathcal{G}(s,t,\upsilon(s,t),\omega(s,t))\big),
		\end{split}
\end{equation}
From (\ref{c8}), (\ref{c9}) and (\ref{c4})we have 
\[	{\mathbb{D}}(\Im(\upsilon_{n},\omega_{n}),\Im(\upsilon,\omega)) \rightarrow 0~ \text{as} ~n\rightarrow \infty.\] 
Step 2. We show that $\Im$ is relatively compact. Let $(\varkappa_{1},y_{1}),(\varkappa_{2},y_{2}) \in \Lambda$ such that $\varkappa_{1}<\varkappa_{2}$ and $y_{1}<y_{2}$. Then we have
\begin{equation*}
	\begin{split}
		&{\mathbb{D}}(\Im_{1}(\upsilon(\varkappa_{2},y_{2}),\omega(\varkappa_{2},y_{2})),\Im_{1}(\upsilon(\varkappa_{1},y_{1}),\omega(\varkappa_{1},y_{1}))) \\&\leq {\mathbb{D}}(h(\varkappa_{2},y_{2}),h(\varkappa_{1},y_{1}))+ \frac{\rho_{1}^{1-\varphi_{1}}\rho_{2}^{1-\varphi
				_2}}{\Gamma(\varphi_{1})\Gamma(\varphi_{2})} \times\\&{\mathbb{D}}\bigg(\int_{0}^{\varkappa_{2}}\int_{0}^{y_{2}}\frac{s^{\rho_{1}-1}t^{\rho_{2}-1}\mathcal{F}(s,t,\upsilon,\omega)dtds}{(\varkappa_{2}^{\rho_{1}}-s^{\rho_{1}})^{1-\varphi_{1}}(y_{2}^{\rho_{2}}-t^{\rho_{2}})^{1-\varphi_{2}}}  ,\int_{0}^{\varkappa_{1}}\int_{0}^{y_{1}}\frac{s^{\rho_{1}-1}t^{\rho_{2}-1}\mathcal{F}(s,t,\upsilon,\omega)dtds}{(\varkappa_{1}^{\rho_{1}}-s^{\rho_{1}})^{1-\varphi_{1}}(y_{1}^{\rho_{2}}-t^{\rho_{2}})^{1-\varphi_{2}}} \bigg)
				\end{split}
		\end{equation*}
			\begin{equation*}
				\begin{split}
					\leq&{\mathbb{D}}(h(\varkappa_{2},y_{2}),h(\varkappa_{1},y_{1}))+ \frac{\rho_{1}^{1-\varphi_{1}}\rho_{2}^{1-\varphi
				_2}}{\Gamma(\varphi_{1})\Gamma(\varphi_{2})}\times\\& {\mathbb{D}}\bigg(\int_{0}^{\varkappa_{1}}\int_{0}^{y_{1}}\frac{s^{\rho_{1}-1}t^{\rho_{2}-1}\mathcal{F}(s,t,\upsilon,\omega)dtds}{(\varkappa_{2}^{\rho_{1}}-s^{\rho_{1}})^{1-\varphi_{1}}(y_{2}^{\rho_{2}}-t^{\rho_{2}})^{1-\varphi_{2}}}  ,\int_{0}^{\varkappa_{1}}\int_{0}^{y_{1}}\frac{s^{\rho_{1}-1}t^{\rho_{2}-1}\mathcal{F}(s,t,\upsilon,\omega)dtds}{(\varkappa_{1}^{\rho_{1}}-s^{\rho_{1}})^{1-\varphi_{1}}(y_{1}^{\rho_{2}}-t^{\rho_{2}})^{1-\varphi_{2}}} \bigg)\\&+{\mathbb{D}}\bigg(\int_{\varkappa_{1}}^{\varkappa_{2}}\int_{0}^{y_{1}}\frac{s^{\rho_{1}-1}t^{\rho_{2}-1}\mathcal{F}(s,t,\upsilon,\omega)dtds}{(\varkappa_{2}^{\rho_{1}}-s^{\rho_{1}})^{1-\varphi_{1}}(y_{2}^{\rho_{2}}-t^{\rho_{2}})^{1-\varphi_{2}}},0\bigg)\\&+{\mathbb{D}}\bigg(\int_{0}^{\varkappa_{1}}\int_{y_{1}}^{y_{2}}\frac{s^{\rho_{1}-1}t^{\rho_{2}-1}\mathcal{F}(s,t,\upsilon,\omega)dtds}{(\varkappa_{2}^{\rho_{1}}-s^{\rho_{1}})^{1-\varphi_{1}}(y_{2}^{\rho_{2}}-t^{\rho_{2}})^{1-\varphi_{2}}},0\bigg)\\&+{\mathbb{D}}\bigg(\int_{\varkappa_{1}}^{\varkappa_{2}}\int_{y_{1}}^{y_{2}}\frac{s^{\rho_{1}-1}t^{\rho_{2}-1}\mathcal{F}(s,t,\upsilon,\omega)dtds}{(\varkappa_{2}^{\rho_{1}}-s^{\rho_{1}})^{1-\varphi_{1}}(y_{2}^{\rho_{2}}-t^{\rho_{2}})^{1-\varphi_{2}}},0\bigg)
	\end{split}
\end{equation*}
\begin{equation*}
	\begin{split}
		\leq&
		{\mathbb{D}}(h(\varkappa_{2},y_{2}),h(\varkappa_{1},y_{1}))+ \frac{\mathcal{M}_{1}\rho_{1}^{1-\varphi_{1}}\rho_{2}^{1-\varphi
				_2}}{\Gamma(\varphi_{1})\Gamma(\varphi_{2})}\times\\&\bigg[\int_{0}^{\varkappa_{1}}\int_{0}^{y_{1}}\big[\frac{s^{\rho_{1}-1}t^{\rho_{2}-1}}{(\varkappa_{2}^{\rho_{1}}-s^{\rho_{1}})^{1-\varphi_{1}}(y_{2}^{\rho_{2}}-t^{\rho_{2}})^{1-\varphi_{2}}} -\frac{s^{\rho_{1}-1}t^{\rho_{2}-1}}{(\varkappa_{1}^{\rho_{1}}-s^{\rho_{1}})^{1-\varphi_{1}}(y_{1}^{\rho_{2}}-t^{\rho_{2}})^{1-\varphi_{2}}}\big]dtds\\&+\int_{\varkappa_{1}}^{\varkappa_{2}}\int_{0}^{y_{1}}\frac{s^{\rho_{1}-1}t^{\rho_{2}-1}dtds}{(\varkappa_{2}^{\rho_{1}}-s^{\rho_{1}})^{1-\varphi_{1}}(y_{2}^{\rho_{2}}-t^{\rho_{2}})^{1-\varphi_{2}}}+\int_{0}^{\varkappa_{1}}\int_{y_{1}}^{y_{2}}\frac{s^{\rho_{1}-1}t^{\rho_{2}-1}dtds}{(\varkappa_{2}^{\rho_{1}}-s^{\rho_{1}})^{1-\varphi_{1}}(y_{2}^{\rho_{2}}-t^{\rho_{2}})^{1-\varphi_{2}}}\\&+\int_{\varkappa_{1}}^{\varkappa_{2}}\int_{y{1}}^{y_{2}}\frac{s^{\rho_{1}-1}t^{\rho_{2}-1}dtds}{(\varkappa_{2}^{\rho_{1}}-s^{\rho_{1}})^{1-\varphi_{1}}(y_{2}^{\rho_{2}}-t^{\rho_{2}})^{1-\varphi_{2}}}\bigg]\\\leq&{\mathbb{D}}(h(\varkappa_{2},y_{2}),h(\varkappa_{1},y_{1}))+ \frac{\mathcal{M}_{1}}{\rho_{1}^{\varphi_{1}}\rho_{2}^{\varphi
				_2}\Gamma(\varphi_{1}+1)\Gamma\varphi_{2}+1)}\big(\varkappa_{2}^{\rho_{1}\varphi_{1}}y_{2}^{\rho_{2}\varphi_{2}}-\varkappa_{1}^{\rho_{1}\varphi_{1}}y_{1}^{\rho_{2}\varphi_{2}}\big).
	\end{split}
\end{equation*}
As $\varkappa_{1}<\varkappa_{2}$ and $y_{1}<y_{2}$ then
\begin{equation}\label{c10}
	{\mathbb{D}}(\Im_{1}(\upsilon(\varkappa_{2},y_{2}),\omega(\varkappa_{2},y_{2})),\Im_{1}(\upsilon(\varkappa_{1},y_{1}),\omega(\varkappa_{1},y_{1})))\longrightarrow 0.
\end{equation} 
In the same way, for $\varkappa_{1}<\varkappa_{2}$ and $y_{1}<y_{2}$
\begin{equation}\label{c11}
	{\mathbb{D}}(\Im_{2}(\upsilon(\varkappa_{2},y_{2}),\omega(\varkappa_{2},y_{2})),\Im_{2}(\upsilon(\varkappa_{1},y_{1}),\omega(\varkappa_{1},y_{1})))\longrightarrow 0.
\end{equation} 
 From (\ref{c10}), (\ref{c11}) with (\ref{c4}),  
  we have $\Im$ is equicontinuous on $\widetilde{W}_{1}$.

Now, to prove that $\Im$ is relatively compact  we must prove
$\Im$ is level-equicontinuous and a compact-supported.
 Since $\mathcal{F}$ and $\mathcal{G}$ are compact mapping, $\mathcal{F}, \mathcal{G}:\Lambda\times\widetilde{W}_{1}\rightarrow \hat{E_{f}}$ are relatively compact  and from theorem (\ref{s1*}) 
$\mathcal{F}$ and $\mathcal{G}$ are level-equicontinuous. Then for each $\varepsilon>0$, $q\in [0,1]$ there exists $\delta>0$ such that from $|r-q|<\delta$ we get
$$d_{H}(\big[\mathcal{F}(\varkappa,y,\upsilon(\varkappa,y),\omega(\varkappa,y))\big]^{r},\big[\mathcal{F}(\varkappa,y,\upsilon(\varkappa,y),\omega(\varkappa,y))\big]^{q})<\frac{\rho_{1}^{\varphi_{1}}\rho_{2}^{\varphi
		_{2}}\Gamma(\varphi_{1}+1)\Gamma(\varphi_{2}+1)}{2a^{\rho_{1}\varphi_{1}}b^{\rho_{2}\varphi_{2}}}\varepsilon,$$	
	$$d_{H}(\big[\mathcal{G}(\varkappa,y,\upsilon(\varkappa,y),\omega(\varkappa,y))\big]^{r},\big[\mathcal{G}(\varkappa,y,\upsilon(\varkappa,y),\omega(\varkappa,y))\big]^{q})<\frac{\rho_{1}^{\psi_{1}}\rho_{2}^{\psi
			_{2}}\Gamma(\psi_{1}+1)\Gamma(\psi_{2}+1)}{2a^{\rho_{1}\psi_{1}}b^{\rho_{2}\psi_{2}}}\varepsilon$$		
and $$d_{H}(\big[h(\varkappa,y)\big]^{r},\big[h(\varkappa,y)\big]^{q}))<\frac{\varepsilon}{2},~~ d_{H}(\big[\mathfrak{g}(\varkappa,y)\big]^{r},\big[\mathfrak{g}(\varkappa,y)\big]^{q}))<\frac{\varepsilon}{2}. $$ Hence, when $|r-q|<\delta$ we have
\begin{equation*}
	\begin{aligned}
		&d_{H}(\big[\Im_{1}(\upsilon,\omega)\big]^{r},\big[\Im_{1}(\upsilon,\omega)\big]^{q})\\&\leq d_{H}(\big[h(\varkappa,y)\big]^{r},\big[h(\varkappa,y)\big]^{q}))\\&+\frac{\rho_{1}^{1-\varphi_{1}}\rho_{2}^{1-\varphi	_{2}}}{\Gamma(\varphi_{1})\Gamma(\varphi_{2})} \int_{0}^{\varkappa}\int_{0}^{y}\frac{s^{\rho_{1}-1}t^{\rho_{2}-1}{d_{H}}(\big[\mathcal{F}(s,t,\upsilon,\omega)\big]^{r},\big[\mathcal{F}(s,t,\upsilon,\omega)\big]^{q})}{(\varkappa^{\rho_{1}}-s^{\rho_{1}})^{1-\varphi_{1}}(y^{\rho_{2}}-t^{\rho_{2}})^{1-\varphi_{2}}}dtds\\\\&\leq\frac{\varepsilon}{2}+\frac{\rho_{1}\rho_{2}\varepsilon\Gamma(\varphi_{1}+1)\Gamma(\varphi_{2}+1)}{2a^{\varphi_{1}}b^{\varphi_{2}}\Gamma(\varphi_{1})\Gamma(\varphi_{2})}\int_{0}^{\varkappa}\int_{0}^{y}\frac{s^{\rho_{1}-1}t^{\rho_{2}-1}}{(\varkappa^{\rho_{1}}-s^{\rho_{1}})^{1-\varphi_{1}}(y^{\rho_{2}}-t^{\rho_{2}})^{1-\varphi_{2}}}dtds\\&\leq\frac{\varepsilon}{2}+\frac{\rho_{1}\rho_{2}\varepsilon\Gamma(\varphi_{1}+1)\Gamma(\varphi_{2}+1)}{2a^{\rho_{1}\varphi_{1}}b^{\rho_{2}\varphi_{2}}\Gamma(\varphi_{1})\Gamma(\alpha_{2})}\frac{\varkappa^{\rho_{1}\varphi_{1}}y^{\rho_{2}\varphi_{2}}}{\rho_{1}\rho_{2}\varphi_{1}\varphi_{2}}.
	\end{aligned}
\end{equation*}
This implies
\begin{equation}\label{c12}
	d_{H}(\big[\Im_{1}(\upsilon,\omega)\big]^{r},\big[\Im_{1}(\upsilon,\omega)\big]^{q})\leq \varepsilon.
\end{equation}
Likewise, we obtain
\begin{equation}\label{c13}
	d_{H}(\big[\Im_{2}(\upsilon,\omega)\big]^{r},\big[\Im_{2}(\upsilon,\omega)\big]^{q})\leq \varepsilon.
\end{equation}
Based on equations (\ref{c12}), (\ref{c13}) and (\ref{c4}) it can be inferred that for each small $\varepsilon > 0$, there exists $\delta > 0$
such that
$$d_{H}(\big[\Im(\upsilon,\omega)\big]^{r},\big[\Im(\upsilon,\omega)\big]^{q})\leq \varepsilon.$$
Thus, $\Im$ is level-equicontinuous.

Since $\mathcal{F}$ and $\mathcal{G}$ are relatively compact,
it can be seen from theorem (\ref{s1*}) that $\mathcal{F}$ and $\mathcal{G}$ are compact-supported and level-equicontinuous
on $[0, 1]$. Hence, there exists a compact subsets $M_{1}^{*},M_{2}^{*} \subset \mathbb{R}$ such that
   $$\big[\mathcal{F} (\varkappa,y,\upsilon(\varkappa,y),\omega(\varkappa,y)\big]^{0} \subseteq M_{1}^{*} ~~and ~~\big[\mathcal{G} (\varkappa,y,\upsilon(\varkappa,y),\omega(\varkappa,y)\big]^{0} \subseteq M_{2}^{*}.$$
Also, since $h,\mathfrak{g}$ are compact-supported, then from (\ref{s00}) there exists a compact subsets  $M_{11}^{*}, M_{22}^{*} \subset \mathbb{R}$ such that
$\big[h(\varkappa, y)\big]^{0} \subseteq M_{11}^{*}$ and $\big[\mathfrak{g}(\varkappa, y)\big]^{0} \subseteq M_{22}^{*}$ for every $(\varkappa,y,\upsilon,\omega) \in \Lambda\times\widetilde{W}_{1}.$ Thus, we obtain
\begin{equation*}
	\begin{aligned}
		&\big[\Im_{1}(\upsilon,\omega)\big]^{0}\\&=\bigg[h(\varkappa,y)+\frac{\rho_{1}^{1-\varphi_{1}}\rho_{2}^{1-\varphi	_{2}}}{\Gamma(\varphi_{1})\Gamma(\varphi_{2})} \int_{0}^{\varkappa}\int_{0}^{y}\frac{s^{\rho_{1}-1}t^{\rho_{2}-1}\mathcal{F}(s,t,\upsilon,\omega)}{(\varkappa^{\rho_{1}}-s^{\rho_{1}})^{1-\varphi_{1}}(y^{\rho_{2}}-t^{\rho_{2}})^{1-\varphi_{2}}}dtds\bigg]^{0}\\\\&=\big[h(\varkappa, y)\big]^{0}+\frac{\rho_{1}^{1-\varphi_{1}}\rho_{2}^{1-\varphi	_{2}}}{\Gamma(\varphi_{1})\Gamma(\varphi_{2})} \int_{0}^{\varkappa}\int_{0}^{y}\frac{s^{\rho_{1}-1}t^{\rho_{2}-1}\big[\mathcal{F} (\varkappa,y,\upsilon,\omega)\big]^{0}}{(\varkappa^{\rho_{1}}-s^{\rho_{1}})^{1-\varphi_{1}}(y^{\rho_{2}}-t^{\rho_{2}})^{1-\varphi_{2}}}dtds\\\\&\subseteq  M_{11}^{*}+\frac{ M_{1}^{*}\varkappa^{\rho_{1}\varphi_{1}}y^{\rho_{2}\varphi_{2}}}{\rho_{1}^{\varphi_{1}}\rho_{2}^{\varphi_{2}}\Gamma(\varphi_{1}+1)\Gamma(\varphi_{2}+1)}.
	\end{aligned}
\end{equation*}
Since $\varkappa^{\rho_{1}\alpha_{1}}y^{\rho_{2}\alpha_{2}}$ is bounded on $\Lambda$, then there exists a compact set $\mathcal{N}_{1}^{*} \subseteq\mathbb{R}$ such that\\ 
$\big[\Im_{1}(\upsilon(\varkappa,y),\omega(\varkappa,y))\big]^{0}\subseteq \mathcal{N}_{1}^{*}$
Hence, $\Im_{1}$ is compact-supported. \\
In a like manner, $\big[\Im_{2}(\upsilon,\omega)\big]^{0} \subseteq  M_{22}^{*}+\frac{ M_{2}^{*}\varkappa^{\rho_{1}\beta_{1}}y^{\rho_{2}\beta_{2}}}{\rho_{1}^{\beta_{1}}\rho_{2}^{\beta_{2}}\Gamma(\beta_{1}+1)\Gamma(\beta_{2}+1)}$ hold, then there exists a compact set $\mathcal{N}_{2}^{*}\subseteq\mathbb{R}$ such that 
$\big[\Im_{2}(\upsilon(\varkappa,y),\omega(\varkappa,y))\big]^{0}\subseteq \mathcal{N}_{2}^{*}$. Hence $\Im_{2}$ is compact-supported.
From (\ref{c0}), it follows that
$$\big[\Im(\upsilon,\omega)\big]^{0} \subseteq \{\mathcal{N}_{i}^{*}\subseteq \mathbb{R} |\mathcal{N}_{i}^{*}\subseteq\mathcal{N}_{j}^{*}, i,j=1,2\},$$
which shows that $\Im$ is compact-supported. And by Ascoli-Arzelá theorem $\Im$ is relatively compact.\\ According to the above steps with Schouder's theorem , we deduce that $\Im$ has at least one fixed point, which is a solution to  the $\mathcal{CK}$ coupled system (\ref{i2}).
\end{proof}
Now in our next theorem, we discuss the uniqueness results for the problem (\ref{i2}).
\begin{theorem} \label{c13*}
	Let $(\upsilon,\omega), (\hat{\upsilon},\hat{\omega}) \in C[{\boldsymbol{\varOmega}},\hat{E_{f}}] \times C[{\boldsymbol{\varOmega}},\hat{E_{f}}]$, for all $(\varkappa,y) \in{\boldsymbol{\varOmega}}$ there exists a  constants $R_{i}>0 (i=1,2)$ such that
	\[{\mathbb{D}}(\mathcal{F}(\varkappa,y,\upsilon,\omega),\mathcal{F}(\varkappa,y,\hat{\upsilon},\hat{\omega}))\leq R_{1}\big(~{\mathbb{D}}(\upsilon,\hat{\upsilon})+ ~{\mathbb{D}}(\omega,\hat{\omega})\big), \]
	 \[{\mathbb{D}}(\mathcal{G}(\varkappa,y,\upsilon,\omega),\mathcal{G}(\varkappa,y,\hat{\upsilon},\hat{\omega}))\leq R_{2}\big(~{\mathbb{D}}(\upsilon,\hat{\upsilon})+ ~{\mathbb{D}}(\omega,\hat{\omega})\big). \]
	Then the problem(\ref{i2}) has a unique solution if 
	\[\varXi_{1}=\frac{R_{1}}{\Gamma(\varphi_{1}+1)\Gamma(\varphi_{2}+1)}
	\bigg(\frac{a^{\rho_{1}}}{\rho_{1}}\bigg)^{\varphi_{1}}\bigg(\frac{b^{\rho_{2}}}{\rho_{2}}\bigg)^{\varphi_{2}}<1,\]
	\[\varXi_{2}=\frac{R_{2}}{\Gamma(\psi_{1}+1)\Gamma(\psi_{2}+1)}
	\bigg(\frac{a^{\rho_{1}}}{\rho_{1}}\bigg)^{\beta_{1}}\bigg(\frac{b^{\rho_{2}}}{\rho_{2}}\bigg)^{\psi_{2}}<1,\] and
\[\varXi^{*}= \binom{\varXi_{1}}{\varXi_{2}}<1.\] 
\end{theorem}
\begin{proof}
	Let $\Im$ be an operator defined in Theorem (\ref{c3}). We will prove that $\Im$ has a unique  fuzzy solution.
	Let $(\upsilon,\omega), (\hat{\upsilon},\hat{\omega}) \in C[\Lambda]\times C[\Lambda]$, for all $(\varkappa,y) \in C[\Lambda]$. Then
	\begin{equation*}
		\begin{aligned}
			&{\mathbb{D}}(\Im_{1}((\upsilon,\omega)(\varkappa,y)),\Im_{1}((\hat{\upsilon},\hat{\omega})(\varkappa,y))) \\&\leq\frac{\rho_{1}^{1-\varphi_{1}}\rho_{2}^{1-\varphi	_{2}}}{\Gamma(\varphi_{1})\Gamma(\varphi_{2})} \int_{0}^{\varkappa}\int_{0}^{y}\frac{s^{\rho_{1}-1}t^{\rho_{2}-1}{\mathbb{D}}(\mathcal{F}(s,t,\upsilon,\omega),\mathcal{F}(s,t,\hat{\upsilon},\hat{\omega}))}{(\varkappa^{\rho_{1}}-s^{\rho_{1}})^{1-\varphi_{1}}(y^{\rho_{2}}-t^{\rho_{2}})^{1-\varphi_{2}}}dtds\\ &\leq \frac{R_{1}\rho_{1}^{1-\varphi_{1}}\rho_{2}^{1-\varphi	_{2}}\big({H^{*}}(\upsilon,\hat{\upsilon})+ {H^{*}}(\omega,\hat{\omega})\big)}{\Gamma(\varphi_{1})\Gamma(\varphi_{2})} \int_{0}^{\varkappa}\int_{0}^{y}\frac{s^{\rho_{1}-1}t^{\rho_{2}-1}}{(\varkappa^{\rho_{1}}-s^{\rho_{1}})^{1-\varphi_{1}}(y^{\rho_{2}}-t^{\rho_{2}})^{1-\varphi_{2}}}dtds\\&\leq  \frac{R_{1} }{\Gamma(\varphi_{1}+1)\Gamma(\varphi_{2}+1)}\bigg(\frac{a^{\rho_{1}}}{\rho_{1}}\bigg)^{\varphi_{1}}\bigg(\frac{b^{\rho_{2}}}{\rho_{2}}\bigg)^{\varphi_{2}}\big({H^{*}}(\upsilon,\hat{\upsilon})+ {H^{*}}(\omega,\hat{\omega})\big).
		\end{aligned}
	\end{equation*} 
	Consequently,
	\begin{equation}\label{c14}
	H^{*}(\Im_{1}((\upsilon,\omega)(\varkappa,y)),\Im_{1}((\hat{\upsilon},\hat{\omega})(\varkappa,y))) \leq\varXi_{1} \big({H^{*}}(\upsilon,\hat{\upsilon})+ {H^{*}}(\omega,\hat{\omega})\big).
	\end{equation}
 Similarly, we get
 	\begin{equation}\label{c15}
 	H^{*}(\Im_{2}((\upsilon,\omega)(\varkappa,y)),\Im_{2}((\hat{\upsilon},\hat{\omega})(\varkappa,y))) \leq\varXi_{2} \big({H^{*}}(\upsilon,\hat{\upsilon})+ {H^{*}}(\omega,\hat{\omega})\big).
 \end{equation}
Therefore, 
\[	H^{*}(\Im((\upsilon,\omega)(\varkappa,y)),\Im((\hat{\upsilon},\hat{\omega})(\varkappa,y)))\leq\varXi^{*} \big({H^{*}}(\upsilon,\hat{\upsilon})+ {H^{*}}(\omega,\hat{\omega})\big),\]
for all $(\upsilon,\omega), (\hat{\upsilon},\hat{\omega}) \in C[\Lambda]\times C[\Lambda]$.
	Hence, $\Im$ is a contraction mapping. Thus, the $\mathcal{CK}$ system (\ref{i2}) has a unique solution. 
\end{proof} 
	\begin{example}
		Consider the following  coupled system of fuzzy fractional PDEs
			\begin{equation}
			\label{5e1}
			\left\{
			\begin{array}{cc}
				^{C}\mathfrak{D}_{0^{+}}^{\frac{1}{2},\frac{3}{2}}\upsilon(\varkappa,y) = \frac{\varkappa y [\upsilon(\varkappa,y)+\omega(\varkappa,y)]}{2(1+\upsilon(\varkappa,y)+\omega(\varkappa,y))},  \\\\
				^{C}\mathfrak{D}_{0^{+}}^{\frac{2}{3},\frac{3}{2}}\omega(\varkappa,y,) = \frac{8(\Gamma(\frac{2}{3}))^{2}}{9e^{\varkappa+y+2}}[\upsilon(\varkappa,y)+\omega(\varkappa,y)],
			\end{array}\right .
		\end{equation}
		with the following initial conditions:
		\begin{equation}
			\label{e2}
			\left\{
			\begin{array}{cc}
				\upsilon(\varkappa,0)=\mathcal{K}\varkappa,~~ 
				\upsilon(0,y)=\mathcal{K}y^{2},\\
				\omega(\varkappa,0)=2\mathcal{K}\varkappa,~~ 
				\omega(0,y)=\mathcal{K}y,
			\end{array}\right .
		\end{equation}
		where $(\varkappa,\zeta)\in 
		(0,\frac{1}{2}]\times(0,1]$, $\mathcal{K}=(1,2,3)$ is a triangle fuzzy number,\\
		$\mathcal{F}(\varkappa,y,\upsilon(\varkappa,y),\omega(\varkappa,y))=\frac{\varkappa y [\upsilon(\varkappa,y)+\omega(\varkappa,y)]}{2(1+\upsilon(\varkappa,y)+\omega(\varkappa,y))}$ and\\
		$\mathcal{G}(\varkappa,y,\upsilon(\varkappa,y),\omega(\varkappa,y))=\frac{8(\Gamma(\frac{2}{3}))^{2}}{9e^{\varkappa+y+2}}[\upsilon(\varkappa,y)+\omega(\varkappa,y)]$,\\
		For any $\upsilon, \omega, \hat{\upsilon}, \hat{\omega} \in E_{f}$ and $(\varkappa,y)\in(0,\frac{1}{2}]\times(0,1],$ we have 
			\[{\mathbb{D}}(\mathcal{F}(\varkappa,y,\upsilon,\omega),\mathcal{F}(\varkappa,y,\hat{\upsilon},\hat{\omega}))\leq \frac{1}{4}\big(~{\mathbb{D}}(\upsilon,\hat{\upsilon})+ ~{\mathbb{D}}(\omega,\hat{\omega})\big), \] and
		\[{\mathbb{D}}(\mathcal{G}(\varkappa,y,\upsilon,\omega),\mathcal{G}(\varkappa,y,\hat{\upsilon},\hat{\omega}))\leq \frac{8(\Gamma(\frac{2}{3}))^{2}}{9e^{2}}\big(~{\mathbb{D}}(\upsilon,\hat{\upsilon})+ ~{\mathbb{D}}(\omega,\hat{\omega})\big). \]
		Moreover, for $\rho=(\rho_{1},\rho_{2})=\frac{3}{2}, \varphi=(\varphi_{1},\varphi_{2})=\frac{1}{2}$ $\psi=(\psi_{1},\psi_{2})=\frac{2}{3}$, $R_{1}=\frac{1}{4}$ and \\$R_{2}=\frac{8(\Gamma(\frac{2}{3}))^{2}}{9e^{2}}$ we have
		\[\varXi_{1}=\frac{R_{1}}{\Gamma(\varphi_{1}+1)\Gamma(\varphi_{2}+1)}
		\bigg(\frac{a^{\rho_{1}}}{\rho_{1}}\bigg)^{\varphi_{1}}\bigg(\frac{b^{\rho_{2}}}{\rho_{2}}\bigg)^{\varphi_{2}}\backsimeq 0.35689 <1,\]
		\[\varXi_{2}=\frac{R_{2}}{\Gamma(\psi_{1}+1)\Gamma(\psi_{2}+1)}
		\bigg(\frac{a^{\rho_{1}}}{\rho_{1}}\bigg)^{\psi_{1}}\bigg(\frac{b^{\rho_{2}}}{\rho_{2}}\bigg)^{\psi_{2}}\backsimeq 0.07882 <1.\] 
		Hence the conditions of theorem (\ref{c13*}) are satisfied.Thus the coupled system given by (\ref{5e1})-(\ref{e2}) has a unique solution  on $\boldsymbol{\varOmega}=(0,\frac{1}{2}]\times (0,1]$. 
	\end{example}
	\section {Conclusion}  
	\ \ \ Study of properties of fractional partial differential equations has important implications for various fields, including physics, biology, finance, and control engineering. The incorporation of fuzzy logic and fractional calculus in the modeling of real-world problems with uncertainty and memory effects is crucial for accurately representing and understanding complex systems. In this paper we  have extended the concept of fuzzy Caputo Katugampola calculus for one-variable functions to be applicable for fuzzy-valued multivariable functions.  By utilizing Schauder's fixed point theorem and certain of notions of relative compactness in fuzzy number spaces, we  have studied the existence and uniqueness of  solutions for both single and coupled systems of fuzzy partial differential equations under Caputo Katugampola derivatives. 

\end{document}